\documentclass[11pt,twoside]{article}

\usepackage[square,numbers]{natbib}

\iffalse %%%%%%%%%%%%%%%%%%%%%%%%%%%%%%%%%%%%%%%%%%%%%%%%%%%%%%%%%%%%%%%%%%%%%%%%
\usepackage[autostyle]{csquotes}
\usepackage[
    backend=biber,
    style=numeric,
    firstinits = true,
    url=false, 
    doi=true,
    eprint=true
]{biblatex}
\renewbibmacro{in:}{%
  \ifentrytype{article}{}{%
  \printtext{\bibstring{in}\intitlepunct}}}

\DeclareFieldFormat[article,unpublished,book]{pages}{#1}  

\DeclareFieldFormat[article]{title}{#1}
\DeclareFieldFormat[unpublished]{title}{#1}
\DeclareFieldFormat[misc]{title}{#1}
\fi %%%%%%%%%%%%%%%%%%%%%%%%%%%%%%%%%%%%%%%%%%%%%%%%%%%%%%%%%%%%%%%%%%%%%%%%

\newcommand*{\doi}[1]{\href{http://dx.doi.org/#1}{doi: #1}}

\usepackage{lipsum} % Package to generate dummy text throughout this template

\usepackage{microtype} % Slightly tweak font spacing for aesthetics

\usepackage{lmodern}
\usepackage[sans]{dsfont}
\usepackage[latin1]{inputenc}
\usepackage{color}
\usepackage{delimdelim}

\usepackage[hmarginratio=1:1,top=32mm,columnsep=20pt]{geometry} % Document margins
\usepackage{booktabs} % Horizontal rules in tables
\usepackage{hyperref} % For hyperlinks in the PDF

\usepackage{paralist} % Used for the compactitem environment which makes bullet points with less space between them

\usepackage{abstract} % Allows abstract customization
\newcommand{\ep}{\epsilon}
\newcommand{\m}{\mbox{d}}
\newcommand{\la}{\left\langle}
\newcommand{\ra}{\right\rangle}

\definecolor{fu}{RGB}{0, 0, 0}
\definecolor{gr}{RGB}{21, 194, 21}

\usepackage{titlesec} % Allows customization of titles
\usepackage{fancyhdr} % Headers and footers
\usepackage{graphicx}
\usepackage{amsmath}
\usepackage{amssymb}
\usepackage{amsthm}
\usepackage{subfigure}
\usepackage{textcomp}
\usepackage{listings}
\usepackage{setspace}
\usepackage{graphicx}
\usepackage{epstopdf}
\usepackage{eurosym}
\usepackage{inputenc}
\usepackage{inputenc}
\usepackage{appendix}
\usepackage{tabularx}
\usepackage{array}
\usepackage{makeidx}
\usepackage{etoolbox}
\usepackage{mathtools}
\usepackage{empheq}
\usepackage{mathrsfs}
\usepackage{marginnote}
\usepackage{caption}

\newcolumntype{L}{>{$}c<{$}}

\DeclareGraphicsExtensions{.eps}
\DeclareGraphicsExtensions{.png}
\theoremstyle{plain}
\newtheorem{theorem}{Theorem}[section]
\newtheorem{prop}[theorem]{Proposition}

\newtheorem*{assumg*}{Assumption~(G)}

\newtheorem*{assumng*}{Assumption~(NG)}
\newtheorem*{assumng1*}{Assumption~(NG-bis)}

\theoremstyle{definition}

\newtheorem{rem}[theorem]{Remark}
\hypersetup{citecolor=black}

 {\proof[#1]\leftskip=0.0cm\rightskip=0.0cm}
 {\endproof}

\label{mathrefs}

\newlength\tindent
\title{\vspace{-5mm}\fontsize{24pt}{10pt} \huge{A priori positivity of solutions to a non-conservative stochastic thin-film equation\vspace{1.0 pc}}} % Article title

\author{
\textsc{Federico Cornalba}
\vspace{1.0 pc}\\ \footnotesize{Department of Mathematical Sciences, University of Bath, Bath, BA2 7AY, United Kingdom}\\ \footnotesize{email: \tt{F.Cornalba@bath.ac.uk}}
\vspace{-0mm}
}
\date{}
\begin{document}

\maketitle % Insert title
\thispagestyle{empty}

\pagestyle{fancy} % All pages have headers and footers

%----------------------------------------------------------------------------------------
%	ABSTRACT
%----------------------------------------------------------------------------------------

\renewenvironment{abstract}
 {\small
  \begin{center}
  \bfseries \abstractname\vspace{-0.0 pc}\vspace{0pt}
  \end{center}
  \list{}{%
    \setlength{\leftmargin}{10mm}% <---------- CHANGE HERE
    \setlength{\rightmargin}{\leftmargin}%
  }%
  \item\relax}
 {\endlist}

 \newenvironment{thanks}
 {%\small
  %\begin{center}
%  \bfseries \abstractname\vspace{-0.0 pc}\vspace{0pt}
  %\end{center}
  \list{}{%
    \setlength{\leftmargin}{0mm}% <---------- CHANGE HERE
    \setlength{\rightmargin}{\leftmargin}%
  }%
  \item\relax}
 {\endlist}

\begin{abstract}

Stochastic conservation laws are often challenging when it comes to proving existence of non-negative solutions. In a recent work by J. Fischer and G. Gr\"un (2018, \emph{Existence of positive solutions to stochastic thin-film equations}, SIAM J. Math. Anal.), existence of positive martingale solutions to a conservative stochastic thin-film equation is established in the case of quadratic mobility. 
In this work, we focus on a larger class of mobilities (including the linear one) for the thin-film model. In order to do so, we need to introduce nonlinear source potentials, thus obtaining a non-conservative version of the thin-film equation. For this model, we assume the existence of a sufficiently regular local solution (i.e., defined up to a stopping time $\tau$) and, by providing suitable conditions on the source potentials and the noise, we prove that such solution can be extended up to any $T>0$ and that it is positive with probability one.
%In this work, we focus on a non-conservative version of the stochastic thin-film equation which allows for specific classes of nonlinear mobilities and source potentials. 
%We provide conditions on these potentials and on the stochastic noise in order to prove positivity of solutions up to any chosen time $T>0$, provided that we assume existence of sufficiently regular solutions up to a random time $\tau\leq T$. 
A thorough comparison with the aforementioned reference work is provided. 

  

\end{abstract}

\section{Introduction}

We are interested in stochastic equations driven by random noise in spatial divergence form. A wide class of these equations can be written as %{\bfseries \emph{why It\"o}}
\begin{align}\label{eq:2000}
\frac{\partial u}{\partial t}=\nabla\cdot\left(m(u)\nabla\frac{\delta F[u]}{\delta u}\right)+\Gamma(u)+\nabla\cdot\left(\sigma\sqrt{m(u)}\mathcal{W}\right)=:\mathscr{D}_1+\mathscr{D}_2+\mathscr{S},
\end{align}
in the non-negative unknown $u=u(x,t)$, for $x\in D\subset\mathbb{R}^d$ and $t>0$. Equation \eqref{eq:2000} describes the evolution of a system made of a large number of particles. The particles are subject to a gradient-flow dynamics (governed by the free energy $F$ featured in the first drift term $\mathscr{D}_1$), to a nonlinear source (given by $\Gamma(u)\equiv\mathscr{D}_2$), and to mesoscopic thermal fluctuations (stochastic term $\mathscr{S}$, comprising an infinite-dimensional noise $\mathcal{W}$ and a given scaling parameter $\sigma\neq 0$). The evolution of the system is described by the particle density $u$, which is naturally required to be non-negative. %; this intuitively justifies the requirement of non-negativity for $u$. 
The drift component $\mathscr{D}_1$ and the noise term $\mathscr{S}$ satisfy a fluctuation-dissipation relation \cite{Chandler1987a} which can be identified in the powers of the so-called \emph{mobility coefficient} $m(u)$ being 1 in $\mathscr{D}_1$ and $\frac{1}{2}$ in $\mathscr{S}$, respectively.\vspace{0.5 pc}\\
When $m(u)\equiv u$ and $\Gamma\equiv 0$, equation \eqref{eq:2000} is known as the \emph{Dean-Kawasaki} model \cite{Dean1996a,Kawasaki1998a}. %,Lutsko2012a}. 
This model poses hard mathematical challenges, the first of which is proving existence of positive solutions up to some given time $T>0$. The main difficulties in doing so reside in the nature of the stochastic noise $\mathscr{S}$. To start with, this noise lacks Lipschitz properties and spatial regularity. If, in addition, we assume $\mathcal{W}$ to be a space-time white noise (this is a relevant choice in the physics literature), then the only existence result we are aware of is the recent work \cite{Lehmann:2018aa}. More specifically, in the case of $F(u):=(N/2)\int_{D}{u(x)\log(u(x))\m x}$ (corresponding to the Gibbs-Boltzmann entropy functional with pre-factor $N/2>0$), 
%the Gibbs-Boltzmann entropy functional $F(u)=(\alpha/2)\int_{D}{u(x)\log(u(x))\m x}$, 
a unique probability measure-valued solution exists if and only if $N\in\mathbb{N}$; however, in this case, the solution is trivial, and coincides with the empirical measure associated with $N$ independent diffusion processes. \vspace{0.5 pc}\\
Again for $m(u)\equiv u$, and for a specific class of $\Gamma\neq 0$, existence of measure-valued martingale solutions to \eqref{eq:2000} is available in space dimension one, see the work of von Renesse and coworkers \cite{Renesse2009a,Andres2010a,Konarovskyi2018a,Konarovskyi2017a}. These results are based on the application of Dirichlet form methods, as well as on the interaction between drift and noise in the context of the Wasserstein geometry over the space of square-integrable probability measures. We also mention \cite{Cornalba2018regularised} for a high-probability existence and uniqueness result for a regularised version of \eqref{eq:2000}. 
\vspace{0.5 pc}\\
In this work we investigate a priori positivity of solutions, %{\bfseries \emph{(importance)}}, 
up to any chosen time $T>0$, in the specific case of a non-conservative \emph{thin-film} equation
\begin{align}\label{eq:1}
\left\{
     \begin{array}{l}
      \m u = -\nabla\cdot \left(m(u)\nabla\left[\Delta u-W'(u)\right]\right)\m t+\left(h(u)|\nabla u|^2+{\color{fu}g(u)}\right)\m t+\nabla\cdot \left(\sqrt{m(u)}\m \mathcal{W}\right), \\
     u(x,0)=u_0(x)
     \end{array}
     \right.
\end{align}
set on the spatial domain $D:=(0,2\pi)$, on some finite time domain $[0,T]$, and on a probability space $(\Omega,\mathcal{F},\mathbb{P})$. More precisely, we assume the existence of a sufficiently regular local solution to \eqref{eq:1} (i.e., defined up to a random time $\tau\leq T$) and we show that it can be extended up to $T$ while remaining positive with probability one.  
%Solutions are \emph{assumed} to exist (up to a random time $\tau\leq T$) and be sufficiently regular; 
%More precisely, on the spatial domain $D:=(0,2\pi)$, on a probability space $(\Omega,\mathcal{F},\mathbb{P})$, and for a suitable initial datum $u_0\colon D\rightarrow [0,\infty)$, we consider the following equation
%\begin{align}\label{eq:1}
%\left\{
%     \begin{array}{l}
%      \m u = -\nabla\cdot \left(m(u)\nabla\left[\Delta u-W'(u)\right]\right)\m t+\left(h(u)|\nabla u|^2+{\color{fu}g(u)}\right)\m t+\nabla\cdot \left(\sqrt{m(u)}\m \mathcal{W}\right), \\
%     u(x,0)=u_0(x),
%     \end{array}
%     \right.
%\end{align}
%\begin{align}\label{eq:2001}c
%\left\{
%     \begin{array}{l}
%      \m u = -\nabla\cdot \left(m(u)\nabla\left[\Delta u-W'(u)\right]\right)\m t+\nabla\cdot \left(\sqrt{m(u)}\m \mathcal{W}\right), \\
%     u(x,0)=u_0(x),
%     \end{array}
%     \right.
%\end{align}
%
%{\color{red}
%\begin{itemize}
%\item analysis of noise $\nabla\cdot(\sqrt{u}\xi)$. No existence, no positivity. 
%\item analysis of noise $\nabla(|u|\xi)$ done in Fischer-Grun, case $m(u)=u^2$.
%\item no analysis for different mobilities. We try to do that with a priori argument.
%\item existence with Galerkin.
%\item adding non-conservative terms for two main reasons: \emph{positivity} and \emph{non-linearity issues}.
%\end{itemize}
%}
%This work is mainly concerned with the analysis of a non-conservative version of a suitable stochastic one-dimensional thin-film equation. 
Above, $u_0\colon D\rightarrow [0,\infty)$ is a suitable positive initial datum, $\mathcal{W}$ is a noise white in time and coloured in space, $m$ is the mobility coefficient, and $W$, $h$ {\color{fu}and $g$} are given nonlinear source potentials.
These potentials compensate the noise contribution whenever the solution comes close to the singular regimes (these being identified by vanishing or diverging density); this is thoroughly discussed in Sections \ref{s:3} and \ref{s:4}. The precise nature of $\mathcal{W}$, $W$, $h$, $m$, and ${\color{fu}g}$ is stated in Subsection \ref{ss:10} below. We highlight that \eqref{eq:1} fits into the form prescribed by \eqref{eq:2000} with $F(u):=\int_{D}{\left\{|\nabla u(x)|^2/2+W(u(x))\right\}\m x}$ and $\Gamma(u):=h(u)|\nabla u|^2+{\color{fu}g(u)}$. %The need for the non-conservative correction $\Gamma(u)$ will be thoroughly discussed. 
\vspace{0.5 pc}\\
Existence of positive martingale solutions to \eqref{eq:1} has been established in the conservative case (${\color{fu}g}\equiv h \equiv 0$) in \cite{Fischer2018a}, for the case of quadratic mobility $m(u)=u^2$; this mobility results in a linear multiplicative stochastic noise. The case of general polynomial mobility, including the linear case $m(u)=u$ (corresponding to the noise $\mathscr{S}$ featured in the Dean-Kawasaki model), seems hard to study for the conservative thin-film equation, see \cite{Fischer2018a} again. This is why we analyse \eqref{eq:1} for a non-trivial drift component $\Gamma$. %, in a similar spirit to in \cite{Renesse2009a,Andres2010a,Konarovskyi2018a,Konarovskyi2017a}. 
However, our drift component $\Gamma$ is not justified, as in the case of \cite{Renesse2009a,Andres2010a,Konarovskyi2018a,Konarovskyi2017a}, by the aforementioned Wasserstein geometry setting. Instead, it is needed in order to deal with algebraic cancellations arising from the It\^o calculus applied to relevant functionals of the solution, these functionals being primarily associated with positivity of the solution, which is our main interest here. 
%\begin{rem}
We also stress the fact that we only pursue a purely analytical justification of our drift component $\Gamma$, and we consequently neglect any physical modelling at this stage. 

%\end{rem}
%In contrast to \cite{Renesse2009a,Andres2010a,Konarovskyi2018a,Konarovskyi2017a}, we do not rely on the Wasserstein geometry as a justification for the drift correction $\Gamma$. %In addition, we are only concerned with an a priori positivity analysis of the solutions.
%\vspace{0.5 pc}\\
The paper is organised as follows. Subsection \ref{ss:10} contains basic assumptions on the functional setting, on the stochastic noise $\mathcal{W}$, as well as a parametrisation of interest for the relevant nonlinear quantities $m,W,h$, and $g$. %Additional notation is set up in Subsection \ref{ss:12}. 
Section \ref{s:2} contains the two main results of this paper, Proposition \ref{p:1} and Theorem \ref{thm:1}.
More specifically, Theorem \ref{thm:1} (which is also proved in this section) is concerned with positivity of solutions to \eqref{eq:1} up to time $T$, which is our main interest. Its proof builds upon Proposition \ref{p:1}, a technical result whose lengthy proof is the topic of Section \ref{s:3}. Sections \ref{s:4} compares the contents of this paper with the setting and conclusions of \cite{Fischer2018a}. Section \ref{s:5} illustrates the difficulties that one encounters when trying to prove existence of local solutions to \eqref{eq:1} via an approximating Galerkin scheme in the case of general mobility $m$, and also explains why such a scheme is effective in the specific case of quadratic mobility \cite{Fischer2018a}. We summarise our findings and conclusions in Section \ref{ss:6}.

\subsection{Setting and notation}\label{ss:10}
We work in a periodic function setting on $D:=(0,2\pi)$. The noise $\mathcal{W}$ is white in time and coloured in space. Its covariance operator $Q$ is diagonalisable on the eigenfunctions of the Laplace operator on $D$ with periodic boundary conditions. These eigenfunctions are given by the trigonometric family 
\begin{align*}
\left\{ e_r \right\}_{r=0}^{\infty}:= \left\{\frac{1}{\sqrt{2\pi}},\frac{\sin(x)}{\sqrt{\pi}},\frac{\cos(x)}{\sqrt{\pi}},\frac{\sin(2x)}{\sqrt{\pi}},\frac{\cos(2x)}{\sqrt{\pi}},\frac{\sin(3x)}{\sqrt{\pi}},\frac{\cos(3x)}{\sqrt{\pi}},\cdots\right\}.
\end{align*}
%$\left\{1/\sqrt{2\pi}\right\}\cup\left\{\cos(rx)/\sqrt{\pi}\right\}_{r=1}^{\infty}\cup\left\{\sin(rx)/\sqrt{\pi}\right\}_{r=1}^{\infty}$. We denote this family by $\{e_r\}_{{\color{red}r=1}}^{\infty}$, where we have chosen the following ordering
%\begin{align*}
%\frac{1}{\sqrt{2\pi}},\frac{\sin(rx)}{\sqrt{\pi}},\frac{\cos(rx)}{\sqrt{\pi}},\frac{\sin(2rx)}{\sqrt{\pi}},\frac{\cos(2rx)}{\sqrt{\pi}}\cdots
%\end{align*}
%The family of associated eigenvalues $\{\mu_r\}_{r=0}^{\infty}$ is $\{0\}\cup\left\{-r^2\right\}_{r=1}^{\infty}$, where the negative eigenvalues have multiplicity two. 
Using \cite[Proposition 2.1.10]{Prevot2007a}, we write the noise as 
$\mathcal{W}(t,x,\omega)=\sum_{r=0}^{\infty}{\sqrt{\lambda_r} e_r(x)\beta_r(t,\omega)}$, where %$\{e_r\}_{r=0}^{\infty}$ and 
$\{\lambda_r\}_{r=0}^{\infty}$ are the eigenvalues of $Q$ associated with $\{e_r\}_{r=0}^{\infty}$, and $\{\beta_r\}_{r=0}^{\infty}$ is a family of independent Brownian motions. % which is associated with the covariance operator $Q$ defined by
%\begin{align}\label{eq:8}
%Qe_r=\lambda_re_r,\qquad r\in\mathbb{N}.
%\end{align}
We assume the eigenvalues of $Q$ to be rapidly decaying, say $\lambda_r\leq a_1e^{-a_2r}$, where $a_1,a_2>0$, for all $r\in\mathbb{N}_0$.\\
For some $\ep\in(0,1)$, let $A_0:=(0,1-\ep)$, $A_{1}:=[1-\ep,1+\ep]$, $A_{\infty}:=(1+\ep,\infty)$. The mobility $m$ and the functions $h$ and ${\color{fu}g}$ are given by
\begin{align}\label{eq:30}
m(u) :=
\left\{
     \begin{array}{ll}
      \! u^{\gamma_1}, & \mbox{ if } u\in A_0, \\
       \! f_m(u), & \mbox{ if } u\in A_1, \\
       \! u^{\gamma_2}, & \mbox{ if } u\in A_{\infty}, 
     \end{array}
     \right.
     \,\,
h(u):=
\left\{
     \begin{array}{ll}
      \! B_hu^{-p_h}, & \mbox{ if } u\in A_0, \\
      \! f_h(u), & \mbox{ if } u\in A_1, \\
      \! -B_hu^{c_h}, & \mbox{ if } u\in A_{\infty}, \\
     \end{array}
     \right.
     \,\,    
g(u):=
\left\{
     \begin{array}{ll}
      \! B_gu^{-p_g}, & \mbox{ if } u\in A_0, \\
      \! f_g(u), & \mbox{ if } u\in A_1, \\
      \! -B_gu^{c_g}, & \mbox{ if } u\in A_{\infty}, 
     \end{array}
     \right.%\quad%\nonumber\\
%\left\{
%     \begin{array}{rl}
%      Bu^{-p}, & \mbox{ if } u\in A_0, \\
%      Bu^c, & \mbox{ if } u\in A_{\infty},\\
%      f_1(u), & \mbox{ if } u\in A_1,
%     \end{array}
%     \right.%\quad%\nonumber\\
\end{align}
while $W$ is given by $W(u)=u^{-p}$. The functions $m,h,g,$ and $W$ are understood to be infinite when $u\leq 0$. In the above, $p,B_h,p_h,c_h,{\color{fu}B_g,p_g,c_g},\gamma_1$, and $\gamma_2$ are positive constants, while the functions $f_h,f_g$, and $f_m$ are such that $W,h,{\color{fu}g}$, and $m$ belong to $\mathcal{C}^{\infty}(0,\infty)$. It is easy to choose $f_h$ and $f_m$ such that, for some $\delta>0$\begin{subequations}
\label{eq:1000}
\begin{empheq}[left={}]{align}
%  \,\,f''_1(u)\geq 0, & \quad\mbox{for all } u\in A_1,\label{eq:1000a}\\
  \,\,f_m(u)> \delta , & \quad\mbox{for all } u\in A_1,\label{eq:1000c}\\
  \,\,f_h'(u)\leq -\delta B_h, & \quad\mbox{for all } u\in A_1.\label{eq:1000b}
\end{empheq}
\end{subequations}
%\begin{align}%\label{eq:1000}
%     \left\{
%     \begin{array}{ll}\label{eq:1000}
%      f''_1(u)\geq 0, & \mbox{for all } u\in A_1, \vspace{0.3 pc}\\
%     f_2'(u)\leq (B'u^{-p'})', & \mbox{for all } u\in[1-\ep,1],\vspace{0.3 pc}\\
%     f_2'(u)\leq (-B'u^{c'})', & \mbox{for all } u\in[1,1+\ep].
%     \end{array}
%     \right.
%\end{align}
%\begin{description}
%$f''_1(u)\geq 0$ for all $u\in[1-\ep,1+\ep]$.
%$f_2'(u)\leq (B'u^{-p'})'$ for all $u\in[1-\ep,1]$, and $f_2'(u)\leq (-B'u^{c'})'$ for all $u\in[1,1+\ep]$.
%\end{description}
The potentials $W$, $h$, and the mobility $m$ are sketched in Figure \ref{fig:1}, {\color{fu}while the potential $g$ is not sketched (as it is qualitatively identical to $h$}). We defined $h,{\color{fu}g}$ and $m$ piecewise on $A_0$ and $A_\infty$ in order to be able to treat low and large density regimes differently. The definitions on $A_1$ provide smoothness on $(0,\infty)$ for the quantities in \eqref{eq:30}.
%; the region $A_1$ does not have any interesting feature in its own right. \\
%The potential $W(u)$, the ``source potential'' $h(u)$, and the mobility $m(u)$ exhibit different behaviours in the two regions $\{u\leq 1\}$ and $\{u>1\}$. 
Our definitions of $W$, $h$, $g$, and $m$ are justified as follows: the potential $W$ pushes mass away from the repulsive singularity $0$, while obeying the conservation of mass. The source potentials $h$ and ${\color{fu}g}$ introduce mass in the system whenever the density is too low, and remove mass whenever the density is too large. {\color{fu}In the case of $h$}, the rate at which the introduction/removal of mass occurs is proportional to $|\nabla u|^2$. %{\color{red} in a Hamilton-Jacobi-Bellman fashion???}
The mobility accounts for different drift and noise magnitudes in the low and large density regimes. %{\bfseries \emph{gradient flow}}.
%$$
%\mbox{REMOVE }\quad W(u)=u^{-p}\,\,\, u\,\,\,h(u)\propto u^{-p_h}\,\,\,h(u)\propto -u^{c_h}\,\,\,m(u)=u^{\gamma_1}\,\,\,m(u)=u^{\gamma_2}
%$$
\begin{figure}[h]
\begin{center}
\framebox{\includegraphics[width=5.5 cm,height=4.5 cm]{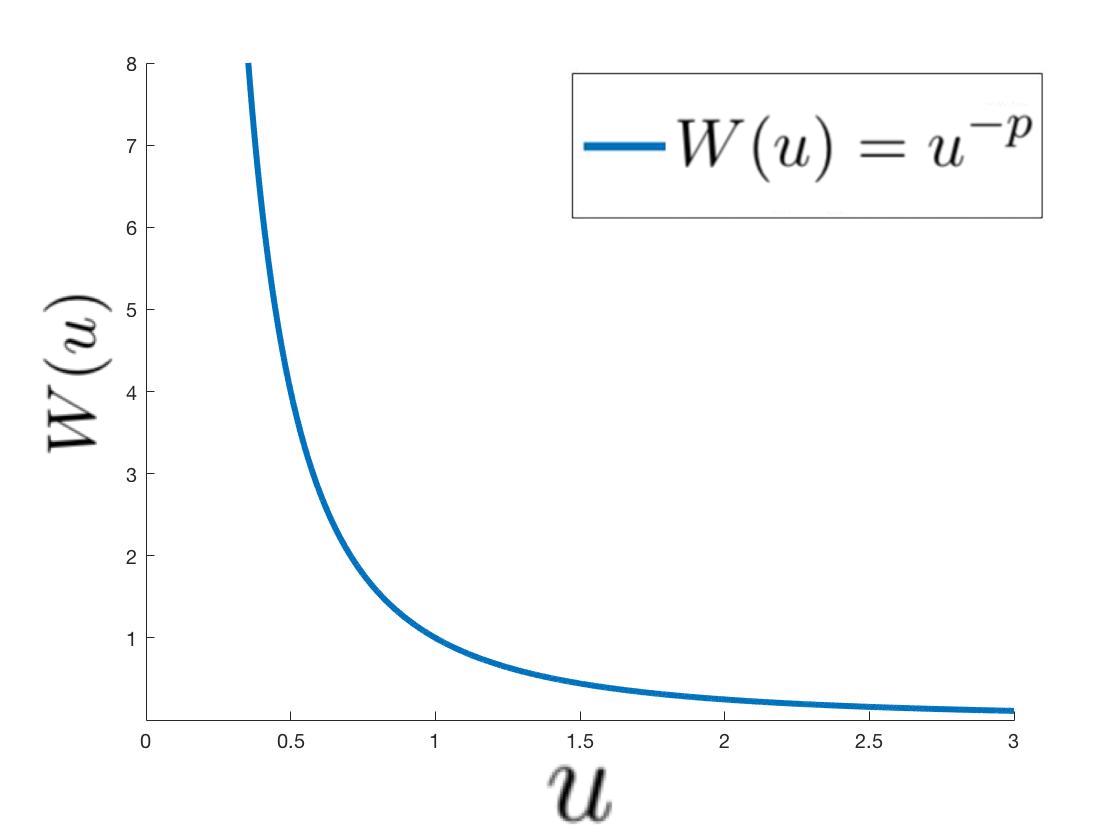}}\hspace{0.5 pc}\framebox{\includegraphics[width=5.5 cm,height=4.5 cm]{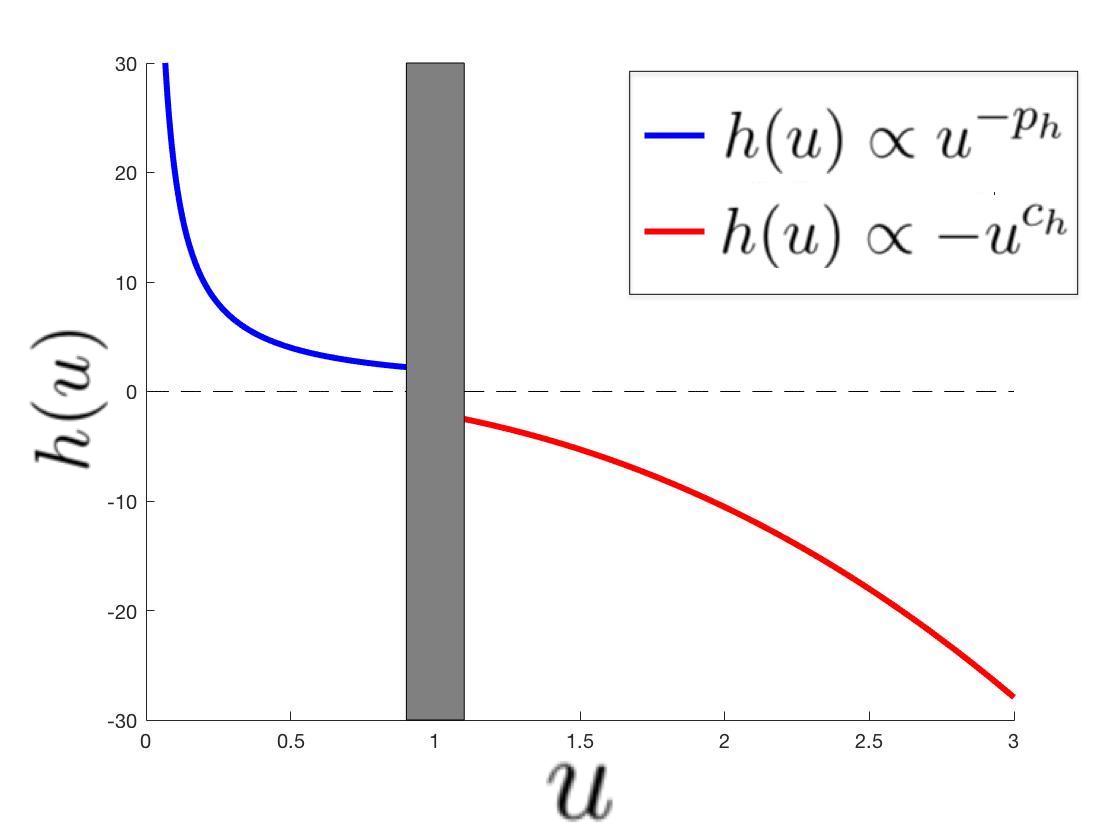}}\hspace{0.5 pc}\framebox{\includegraphics[width=5.5 cm,height=4.5 cm]{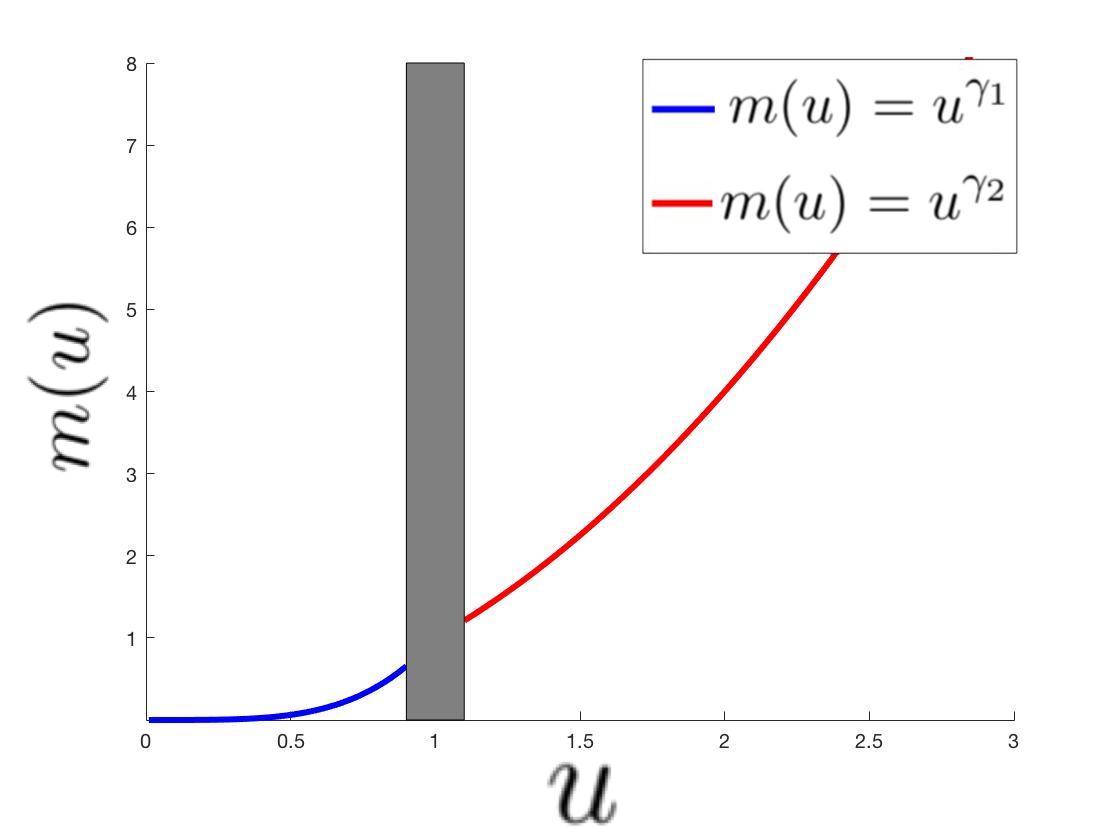}}
\caption{Sketches of $W$ (left), $h$ (centre), and $m$ (right). Plots on $A_1$ are not provided for $h$ and $m$. {\color{fu}The qualitative behaviour of $g$ is identical to that of $h$.}}
\label{fig:1}
\end{center}
\end{figure}
%\begin{rem}\label{rem:1}
%From \eqref{eq:30}, it is immediate to notice that $W$ and $m$ are only $\mathcal{C}^{0}(0,\infty)$, while $h$ is not even continuous. In any case, there is an issue of low regularity only at $u=1$. This will not create any problems in our It\^o calculus computations associated with \eqref{eq:1}, see Sections \ref{s:3}, \ref{s:4}, \ref{s:5} below. As a matter of fact, many computations will be carried by neglecting these singularities at $u=1$. At a later stage, the same computations will be adjusted in order to allow for a $\mathcal{C}^4$-regularisation of $W$, $h$, and $m$, thus making the whole thing rigorous. We have chosen this approach so that we can avoid to deal with an excessive number of technical details all at once. 
%\end{rem}
%{\color{red}
%\begin{rem}
%The term $h(u)|\nabla u|^2$ could be associated with the Hamilton-Jacobi-Bellman equation. Further analysis is required.
%\end{rem}}

%\subsection{Notation}\label{ss:12}

We use the symbol $L^p$ to denote the space $L^{p}(D)$. We use the symbol $W^{s,p}$ to denote the Sobolev space $W_{\mbox{\scriptsize{per}}}^{s,p}(D)$ of $2\pi$-periodic functions on $\mathbb{R}$ having distributional derivates up to order $s$ belonging to $L^p$. We abbreviate $H^{s}:=W^{s,2}$.
For a Hilbert space $V$, we use $\la\cdot,\cdot\ra_V$ and $\|\cdot\|_V$ to denote the $V$-inner product and $V$-norm, respectively. We drop the subscript if $V=L^2$. For a function $u$ depending on space and time, we often write $u(t)$ instead of $u(x,t)$, and we indifferently use the notations $u_x$ and $\nabla u$ to refer to spatial differentiation. Finally, $C$ denotes a generic constant whose value may change from line to line; the dependency of this constant on specific parameters is highlighted whenever relevant. 

\section{A priori positivity of solutions}\label{s:2}

Let $T>0$. We show that, if we assume the existence of a sufficiently regular solution to \eqref{eq:1} up to a random time $\tau\leq T$, this solution can be extended up to $T$ and is positive $\mathbb{P}$-a.s. In order to do so, we need the following auxiliary result.
%1This section is devoted to proving existence and positivity of solutions to \eqref{eq:1} up to any fixed time $T>0$, provided that we \emph{assume} the existence of a sufficiently regular solution to \eqref{eq:1} up a random time $\tau\leq T$. We have the following results.
\begin{prop}\label{p:1}
%Let $D:=(0,2\pi)$. 
Fix $T>0$ and $\beta>2$. Consider an initial datum $u_0\in H^1$ such that $\delta_1<\min_{x\in D}{u_0(x)}$ and $\|u_0\|_{H^1}< \delta_2$, for some $\delta_2>\delta_1>0$, $\mathbb{P}$-a.s. Assume the existence of a strong solution $u$ to \eqref{eq:1} up to a random time $\tau\leq T$. More precisely, we assume the equation below to be satisfied $\mathbb{P}$-a.s., for all $t>0$
\begin{align}\label{eq:1000}
u(t\wedge \tau) & = u_0+\int_{0}^{t\wedge \tau}{\left[-\nabla\cdot \left(m(u)\nabla\left[\Delta u-W'(u)\right]\right)+\left(h(u)|\nabla u|^2+{\color{fu}g(u)}\right)\right]\emph{\m} s} \nonumber\\
 & \quad + \int_{0}^{t\wedge \tau}{\nabla\cdot\left(\sqrt{m(u)}(\cdot)\right)\emph{\m}\mathcal{W}}.
\end{align}
%to hold $\mathbb{P}$-a.s., for all $t>0$. 
We assume that $u$ has $\mathbb{P}$-a.s. continuous paths with respect to the $H^1$-norm, and that
\begin{align}\label{eq:3002}
\mathbb{P}\left(\int_{0}^{\tau}{\!\|\nabla u(s)\|^4_{L^4}\emph{\m}s}<\infty\right)=1,\quad\!\mathbb{P}\left(\int_{0}^{\tau}{\!\|\Delta u(s)\|^2\emph{\m}s}<\infty\right)=1,\quad\!\mean{\int_{0}^{\tau}{\!\left\|u_{xxx}(s)\sqrt{m(u(s))}\right\|^2\!\emph{\m} s}}<\infty.
\end{align}
For all $n\in\mathbb{N}$ such that $n^{-1}<\delta_1$ and $n>\delta_2$, we assume $\tau_n \leq \tau\leq T$, where the stopping time $\tau_n$ is given by
\begin{align}\label{eq:51}
\tau_n:=\inf\left\{t>0:\min_{x\in D}{u(x,t)}\leq n^{-1}\right\}\wedge\inf\left\{t>0:\|u(t)\|_{H^1}\geq n\right\}\wedge T.
\end{align} 
Assume the following conditions
\begin{align}%\label{eq:41}
%\tag{Con2}
%& \gamma_2\leq 2,\\
%\tag{Con3}\label{eq:42}
%& {\color{gr}2\leq\gamma_1\leq2+\beta},\\
%\tag{Con1}\label{eq:43}
\tag{C1}\label{eq:43}
& \sum_{r=0}^{\infty}{\lambda_r}\mbox{ \emph{is small enough}},\\
%\tag{Con2}\label{eq:44}
\tag{C2}\label{eq:44}
& p_h,B_h,c_h\mbox{ \emph{are big enough}},\\
%\tag{Con3}\label{eq:45}
\tag{C3}\label{eq:45}
& {\color{fu}p_g,B_g,c_g\mbox{ \emph{are big enough}}}.
\end{align}
Let $F_1\colon H^1\rightarrow \mathbb{R}\cup\{\infty\}\colon u\mapsto \int_{D}{|u|^{-\beta}}$, let $F_2\colon H^1\rightarrow \mathbb{R}\colon u\mapsto \frac{1}{2}\|u\|^2_{H^1}$, and let $F:=F_1+F_2$. There is a constant $C$ independent of $n$, such that 
\begin{align}\label{eq:50}
%\mean{\|u^{-\theta}(t\wedge \tau_n)\|_{W^{1,1}(0,2\pi)}} \leq C,\qquad \mbox{for all }t\in[0,T],
%\mean{\int_{D}{u^{-\beta}(t\wedge\tau_n)}}+\mean{\|u(t\wedge\tau_n)\|_{H^1}^2} \leq C,\qquad \mbox{for all }t\in[0,T],
\mean{F(u(t\wedge \tau_n))} \leq C,\qquad \mbox{for all }t\in[0,T].
\end{align}
%where $C$ does not depend on $n$.  
\end{prop}
%{\bfseries \emph{ Honest comment about the above result, and validity of It\^o formula}}. 
The proof of Proposition \ref{p:1}, which is quite lengthy and technical, is the content of Section \ref{s:3}. Our main result, which relies on Proposition \ref{p:1}, is the following.
\begin{theorem}\label{thm:1}
Let the assumptions of Proposition \ref{p:1} be satisfied. Then the solution $u$ to \eqref{eq:1000} is defined up to time $T$ and is $\mathbb{P}$-a.s. positive, meaning that
\begin{align*}
\mathbb{P}\left(u(x,t)>0\mbox{ for all }x\mbox{ in }D\mbox{ and for all } t\in[0,T]\right)=1.
\end{align*}
\end{theorem}
\begin{proof}
Define $\theta:=\frac{\beta}{2}-1>0$. The H\"older inequality and the bound $u^{-\theta}\leq u^{-\beta}+1$, valid on $(0,\infty)$, give
\begin{align*}%\label{eq:52}
\|u^{-\theta}(t\wedge\tau_n)\|_{W^{1,1}} & =\int_{D}{|u^{-\theta}(t\wedge\tau_n)|\m x}+\theta\int_{D}{|u^{-\theta-1}(t\wedge\tau_n)\nabla u(t\wedge\tau_n)|\m x}\nonumber\\
& \leq\int_{D}{|u^{-\theta}(t\wedge\tau_n)|\m x}+\theta\left(\int_{D}{|u^{-2(\theta+1)}(t\wedge\tau_n)|\m x}\right)^{1/2}\left(\int_{D}{|\nabla u(t\wedge\tau_n)|^2\m x}\right)^{1/2}\nonumber\\
& \leq C+C\int_{D}{|u^{-\beta}(t\wedge\tau_n)|\m x} + C\|u(t\wedge\tau_n)\|^2_{H^1}\leq C + CF(u(t\wedge\tau_n)).
\end{align*}
This immediately entails, using Proposition \ref{p:1}, that
\begin{align}\label{eq:1001}
\mean{\|u^{-\theta}(t\wedge\tau_n)\|_{W^{1,1}}}\leq C,\qquad \mbox{for all }t\in[0,T],
\end{align}
where $C$ is independent of $n$. Let $t\in[0,T]$. %, and $n\in\mathbb{N}$ such that $n^{-1}<\delta$, where $\delta$ is inherited from the notation of Proposition \ref{p:1}. 
We use the $\mathbb{P}$-a.s. $H^1$-continuity of the paths of $u$, the continuous embedding $W^{1,1}\hookrightarrow C(0,2\pi)$ %and $H^1\subset C(0,2\pi)$ 
(with embedding constant $K_1$), the Chebyshev inequality, and equations \eqref{eq:50} and \eqref{eq:1001} to deduce
\begin{align*}
\mathbb{P}\left(\tau_n<t\right) & \leq \mathbb{P}\left(\min_{x\in D}{|u(t\wedge\tau_n)|}\leq n^{-1}\right) + \mathbb{P}\left(\|u(t\wedge\tau_n)\|_{H^1}\geq n\right) = \mathbb{P}\left(\max_{x\in D}{|u(t\wedge\tau_n)|^{-\theta}}\geq n^{\theta}\right)\\
%& =\mathbb{P}\left(\left\{\min_{x\in D}{|u(t\wedge\tau_n)|}\right\}^{-\theta}\geq n^{\theta}\right) + \mathbb{P}\left(\max_{x\in D}{|u(t\wedge\tau_n)|}\geq n\right)\\
& \quad + \mathbb{P}\left(\|u(t\wedge\tau_n)\|^2_{H^1}\geq n^2\right) \leq \mathbb{P}\left(\|u^{-\theta}(t\wedge \tau_n)\|_{W^{1,1}}\geq K_1^{-1}n^{\theta}\right) + \mathbb{P}\left(\|u(t\wedge\tau_n)\|^2_{H^1}\geq n^2\right)\\
%&  \leq \mathbb{P}\left(\|u^{-\theta}(t\wedge \tau_n)\|_{W^{1,1}}\geq K_1^{-1}n^{\theta}\right) + \mathbb{P}\left(\|u(t\wedge \tau_n)\|_{H^1}\geq K_2^{-1}n\right)\\
& \leq \frac{\mean{\|u^{-\theta}(t\wedge \tau_n)\|_{W^{1,1}}}}{K_1^{-1}n^{\theta}} + \frac{\mean{\|u(t\wedge \tau_n)\|^2_{H^1}}}{n^2} \rightarrow 0%\leq CK_1n^{-\theta}\rightarrow 0
\end{align*}
as $n\rightarrow 0$. This implies that $\mathbb{P}\left(\sup_{n}{\tau_n}=T\right)=1$, and concludes the proof.
\end{proof}

\section{Proof of Proposition \ref{p:1}}\label{s:3}
We split the proof in four parts. In Subsection \ref{ss:1}, we compute and properly bound the It\^o differential of the process $F(u)$ up to time $t\wedge\tau_n$, for any $t\in[0,T]$. In Subsection \ref{ss:2}, we group all the terms from the previously computed It\^o differential into families, each family being characterised by a specific term. Subsections \ref{ss:3} and \ref{ss:4} are concerned with imposing conditions on the parameters $p,B_h,p_h,c_h,{\color{fu}B_g,p_g,c_g},\gamma_1,\gamma_2$, and $\{\lambda_r\}_{r=0}^{\infty}$ in such a way that \eqref{eq:50} is achieved; more specifically, Subsection \ref{ss:3} provides the relevant analysis on $A_0\cup A_{\infty}$, while Subsection \ref{ss:4} consistently extends this analysis on to $A_1$. \vspace{0.5 pc}\\
For notational convenience, we rewrite \eqref{eq:1000} as 
%\begin{align*}
$
\m u = \phi(u(t))\m t + \Phi(u(t))\m \mathcal{W}(t),
$
%\end{align*}
where 
\begin{align*}
\phi(u)=\phi_1(u)+\phi_2(u)+\phi_3(u) & := -\nabla\cdot \left(m(u)\nabla\left[\Delta u-W'(u)\right]\right)+h(u)|\nabla u|^2+{\color{fu}g(u)},\\
\Phi(u)v & := \nabla\cdot\left(\sqrt{m(u)}v\right).
\end{align*}
Integration by parts entails that the component of the stochastic noise of \eqref{eq:1000} along the direction $e_i$, for $i\in\mathbb{N}_0$, is 
\begin{align*}
& \la \int_{0}^{t}{\Phi(u(s))\m \mathcal{W}(s)}, e_i\ra = \la \int_{0}^{t}{\nabla\cdot\left(\sqrt{m(u(s))}\sum_{r=0}^{\infty}{\sqrt{\lambda_r} e_r\m\beta_r(s)}\right)}, e_i\ra \\
& = -\la \int_{0}^{t}{\sqrt{m(u(s))}\sum_{r=0}^{\infty}{\sqrt{\lambda_r} e_r\m\beta_r(s)}}, \nabla e_i\ra = \sum_{r=0}^{\infty}{\int_{0}^{t}{-\la \sqrt{m(u(s))}e_r,\nabla e_i\ra\sqrt{\lambda_r}\m\beta_r(s)}}.
\end{align*}
Thus $\Phi$ can be thought of as an infinite-dimensional noise represented with components given by
\begin{align}\label{eq:9}
\Phi_{i,r}(u(s)):=-\la \sqrt{m(u(s))} e_r,\nabla e_i\ra,\qquad \mbox{for all }i,r\in\mathbb{N}_0.
\end{align}

\subsection{It\^o formula for $F(u(t\wedge\tau_n))$}\label{ss:1}

%\section{It\^o formula applied to $F(u)$}\label{s:3}

%We compute the It\^o formula applied to the functional $F(u)$ defined in \eqref{eq:3}. 
%\begin{rem}
We use the It\^o formula 
\begin{align}\label{eq:4}
G(u(t\wedge\tau_n)) & = G(u(0))+\int_{0}^{t\wedge\tau_n}{\!G_{u}(u(s))\phi(u(s))\m s} + \int_{0}^{t\wedge\tau_n}{\frac{1}{2}\mbox{Tr}\!\left[G_{uu}(u(s))(\Phi(u(s)) Q^{\frac{1}{2}})(\Phi(u(s))Q^{\frac{1}{2}})^T\right]\!\m s}\nonumber\\
& \quad + \int_{0}^{t\wedge\tau_n}{G_{u}(u(s))\Phi(u(s))\m \mathcal{W}(s)}+=:I_1+I_2+I_3+I_4, 
\end{align}
here stated for a real-valued functional $G$ applied to the solution $u$. We can apply \eqref{eq:4} to $G=F_1$ and $G=F_2$ because, up to time $t\wedge\tau_n$, they are both uniformly continuous (along with their first and second derivatives) over bounded sets of $H^1$. We analyse terms $I_2$, $I_3$, and $I_4$ of \eqref{eq:4} for $G=F_1$ and $G=F_2$. Time dependence is often dropped for notational convenience.\vspace{0.5 pc}\\
\emph{Term $I_2$ for $G=F_1$}. %\subsection{Drift contribution for $F_1(u)=\int_{D}{u^{-\beta}}$}\label{ss:1}
The first and second derivatives of $F_1$ are $F_{1,u}(u)v=-\beta\int_{D}{u^{-\beta-1}v\m x}$ and $F_{1,uu}(u)(v_1,v_2)=\beta(\beta+1)\int_{D}{u^{-\beta-2}v_1v_2\m x}$. 
%The It\"o formula reads
%\begin{align}\label{eq:4}
%F_1(t) & = F_1(0)+\int_{0}^{t}{\la F_{1,u}(u(s)), \Phi(u(s)\m \mathcal{W}(s))\ra}+\int_{0}^{t}{\la F_{1,u}(u(s)),\phi(u(s))\ra\m s}\nonumber\\
%& \quad + \frac{1}{2}\int_{0}^{t}{\mbox{Tr}\left[F_{1,uu}(u(s))(\Phi Q^{1/2})(\Phi(u(s))Q^{1/2})^T\right]\m s}\nonumber\\
%& =: I_1+I_2+I_3+I_4. 
%\end{align}
%We focus on $I_3$, and more precisely on $\la F_{1,u}(u(s)),\phi(u(s))\ra$. 
We study the contributions of $\phi_1$, $\phi_2$, and {\color{fu}$\phi_3$} on $F_{1,u}(u)\phi(u)$ separately. We obtain
\begin{align*}
 F_{1,u}(u)\phi_1(u) & = \la-\nabla\cdot \left(m(u)\nabla\left[\Delta u-W'(u)\right]\right), -\beta u^{-\beta-1}\ra = \beta(\beta+1)\la m(u)\nabla[\Delta u-W'(u)],u^{-\beta-2}\nabla u\ra\\
& = \beta(\beta+1)\la \nabla[\Delta u-W'(u)],m(u)u^{-\beta-2}\nabla u\ra=-\beta(\beta+1)\la \Delta u,\nabla(m(u)u^{-\beta-2}\nabla u)\ra +\\
& \quad \quad -\beta(\beta+1)\la W''(u)\nabla u,m(u)u^{-\beta-2}\nabla u\ra\\
& = -\beta(\beta+1)\la \Delta u, m(u)u^{-\beta-2}\Delta u\ra - \beta(\beta+1)\la \Delta u,(m(u)u^{-\beta-2})'|\nabla u|^2 \ra\\
& \quad \quad -\beta(\beta+1)\la W''(u)\nabla u,m(u)u^{-\beta-2}\nabla u\ra.
\end{align*}
We remind the reader of the identity
\begin{align}\label{eq:400}
\la f(u)|\nabla u|^2,\Delta u\ra = -\frac{1}{3}\la f'(u)|\nabla u|^2,|\nabla u|^2\ra,
\end{align}
which is valid for $f\in\mathcal{C}^1(0,\infty)$. We choose $f(u):=(m(u)u^{-\beta-2})'$ and deduce
\begin{align}\label{eq:5}
F_{1,u}(u)\phi_1(u) &= -\beta(\beta+1)\la \Delta u, m(u)u^{-\beta-2}\Delta u\ra\nonumber\\
& \quad + \frac{\beta(\beta+1)}{3}\la (m(u)u^{-\beta-2})''|\nabla u|^2,|\nabla u|^2\ra -\beta(\beta+1)\la W''(u)\nabla u,m(u)u^{-\beta-2}\nabla u\ra.
\end{align}
As for $\phi_2$ {\color{fu}and $\phi_3$}, the contributions are simply 
\begin{align}\label{eq:6}
F_{1,u}(u)\phi_2(u) = \la h(u)|\nabla u|^2,-\beta u^{-\beta-1}\ra,\qquad{\color{fu}F_{1,u}(u)\phi_3(u)} = {\color{fu} \la g(u),-\beta u^{-\beta-1}\ra.}
\end{align}
%\end{rem}
%We split the analysis in four subsections. More precisely, with respect to the notation introduced in \eqref{eq:3}, \eqref{eq:4}: Subsection \ref{ss:1} deals with $I_3$ for $G=F_1$; Subsection \ref{ss:2} deals with $I_3$ for $G=F_2$; Subsection \ref{ss:3} deals with $I_4$ for $G=F_1$, Subsection \ref{ss:4} deals with $I_4$ for $G=F_2$.
%\begin{rem}
%This section, as well as Section \ref{eq:4}, is \emph{not} concerned with $\mathcal{C}^4$-regularised versions of $W,h$ and $m$, see Remark \ref{rem:1} also. 
%\end{rem}
%covering the drift term (i.e., $I_3$ in \eqref{eq:4}) and It\^o correction term (i.e., $I_4$ in \eqref{eq:4}) for the functionals $F_1(u)$, $F_2(u)$ which define $F(u)$. %That is to say, we focus on 
%\vspace{0.5 pc}\\
\emph{Term $I_2$ for $G=F_2$}. The first and second derivatives of $F_2$ are $F_{2,u}(u)v=\la u,v\ra_{H^1}$ and $F_{2,uu}(u)(v_1,v_2)=\la v_1,v_2 \ra_{H^1}$. We study the contributions of $\phi_1$, $\phi_2$, {\color{fu}and $\phi_3$} on $F_{2,u}(u)\phi(u)$ separately. We set $f(u):=m(u)W''(u)$ and we obtain, by relying on \eqref{eq:400} and using integration by parts
\begin{align}\label{eq:21}
& F_{2,u}(u)\phi_1(u) \nonumber\\
& \quad = \la -\nabla\cdot \left(m(u)\nabla\left[\Delta u-W'(u)\right]\right),u\ra_{H^1}=\la -\nabla\cdot \left(m(u)\nabla\left[\Delta u-W'(u)\right]\right),u\ra\nonumber\\
&  \quad\quad +\la \nabla(-\nabla\cdot \left(m(u)\nabla\left[\Delta u-W'(u)\right]\right)),\nabla u\ra\nonumber\\
&  \quad= \la m(u)\nabla[\Delta u-W'(u)],\nabla u\ra+\la \nabla\cdot(m(u)\nabla[\Delta u-W'(u)]),\Delta u\ra\nonumber\\
&  \quad= \la \nabla[\Delta u-W'(u)],m(u)\nabla u\ra-\la m(u)\nabla[\Delta u-W'(u)],u_{xxx}\ra\nonumber\\
&  \quad= -\la \Delta u,m'(u)|\nabla u|^2\ra-\la \Delta u, m(u)\Delta u\ra-\la W''(u)\nabla u,m(u)\nabla u\ra-\la m(u)u_{xxx},u_{xxx} \ra + \la f(u)\nabla u,u_{xxx} \ra\nonumber\\
&  \quad= \frac{1}{3}\la m''(u)|\nabla u|^2,|\nabla u|^2\ra-\la \Delta u, m(u)\Delta u\ra-\la W''(u)\nabla u,m(u)\nabla u\ra\nonumber\\
&  \quad\quad -\la m(u)u_{xxx},u_{xxx} \ra-\la f(u)\Delta u,\Delta u\ra-\la f'(u)|\nabla u|^2,\Delta u\ra\nonumber\\
&  \quad= \frac{1}{3}\la \left[m''(u)+f''(u)\right]|\nabla u|^2,|\nabla u|^2\ra -\la \Delta u, m(u)[1+W''(u)]\Delta u\ra\nonumber\\
&  \quad\quad -\la W''(u)\nabla u,m(u)\nabla u\ra-\la m(u)u_{xxx},u_{xxx}\ra.
\end{align}
The contribution associated with $\phi_2$ is
\begin{align}\label{eq:25}
F_{2,u}(u)\phi_2(u) & = \la h(u)|\nabla u|^2,u \ra_{H^1}=\la h(u)u, |\nabla u|^2\ra+\la \nabla(h(u)|\nabla u|^2), \nabla u\ra\nonumber\\
& = \la h(u)u, |\nabla u|^2\ra - \la h(u)|\nabla u|^2, \Delta u\ra = \la h(u)u, |\nabla u|^2\ra + \frac{1}{3}\la h'(u)|\nabla u|^2, |\nabla u|^2\ra,
\end{align}
{\color{fu} while the contribution associated with $\phi_3$ is
\begin{align}\label{eq:25a}
F_{2,u}(u)\phi_3(u) & = \la g(u),u \ra + \la g'(u)\nabla u, \nabla u\ra.
\end{align}
}
\emph{Term $I_3$ for $G=F_1$}. We rely on \eqref{eq:9} and the expression of $F_{1,uu}$ to compute the It\^o correction  
%\begin{align*}
%\frac{1}{2}\mbox{Tr}\left[F_{1,uu}(u)(\Phi(u)Q^{1/2})(\Phi(u)Q^{1/2})^T\right].
%\end{align*}
\begin{align}\label{eq:10}
& \frac{1}{2}\mbox{Tr}\left[F_{1,uu}(u)(\Phi(u)Q^{1/2})(\Phi(u)Q^{1/2})^T\right] \nonumber\\
& \quad = \beta(\beta+1)\sum_{r=0}^{\infty}{\lambda_r\sum_{s=0}^{\infty}{\sum_{z=0}^{\infty}{\la u^{-\beta-2}e_z,e_s\ra\la \sqrt{m(u)}e_r,e_{s,x}\ra\la \sqrt{m(u)}e_r,e_{z,x}\ra}}}.
\end{align}
%{\color{red} 
\begin{rem}\label{rem:3}
One can convince oneself of the nature of \eqref{eq:10} by thinking of a finite-dimensional equivalent of the problem, formulated in terms of the matrices 
\begin{align}\label{eq:11}
Q_m=\mbox{diag}\left\{\sqrt{\lambda_1},\cdots,\sqrt{\lambda_m}\right\},\quad \left[\Phi_{m}(u)\right]_{i,r}:=-\la \sqrt{m(u)}e_r,\nabla e_i\ra,\quad i,r\in\{0,\cdots,m\},\\
\left[F_{1,uu}(u)\right]_m(e_i,e_r)=\beta(\beta+1)\int_{D}{u^{-\beta-2}e_ie_r\m x},\quad i,r\in\{0,\cdots,m\}\nonumber.
\end{align}
\end{rem}
%}
We bound \eqref{eq:10} by using integration by parts, the Parseval identity in $L^2$ (for the sums over $z$ and $s$),  %namely $\la f_1,f_2\ra=\sum_{k=0}^{\infty}{\la f_1,e_k\ra\la f_1,e_k\ra}$
the rapid decay of $\{\lambda_r\}_{r=0}^{\infty}$, and the fact that 
%\begin{align*}
$\|(\m^k/\m x^k)e_r\|_{L^{\infty}}\leq C_kr^k$ (for the sum over $r$).
%\end{align*} 
%to perform the sum in $r$. 
We obtain %{\bfseries \emph{ Curtain Falb}}
\begin{align}
&\beta(\beta+1)\sum_{r=0}^{\infty}{\lambda_r\sum_{s=0}^{\infty}{\sum_{z=0}^{\infty}{\la u^{-\beta-2}e_z,e_s\ra\la \sqrt{m(u)}e_r,e_{s,x}\ra\la \sqrt{m(u)}e_r,e_{z,x}\ra}}}\nonumber\\
%& \quad = \{\mbox{by parts}\} = \beta(\beta+1)\sum_{r}{\lambda_r\sum_{s}{\sum_{z}{\la u^{-\beta-2}e_z,e_s\ra\la \nabla\left(\sqrt{m(u)}e_r\right),e_{s}\ra\la \nabla\left(\sqrt{m(u)}e_r\right),e_{z}\ra}}}\nonumber\\
& \quad =  \beta(\beta+1)\sum_{r=0}^{\infty}{\lambda_r\sum_{s=0}^{\infty}{\sum_{z=0}^{\infty}{\la u^{-\beta-2}e_z,e_s\ra\la \nabla\left(\sqrt{m(u)}e_r\right),e_{s}\ra\la \nabla\left(\sqrt{m(u)}e_r\right),e_{z}\ra}}}\nonumber\\
%& \quad = \{\mbox{sum in }z\} = \beta(\beta+1)\sum_{r}{\lambda_r\sum_{s}{\la \nabla\left(\sqrt{m(u)}e_r\right),e_{s}\ra\la \nabla\left(\sqrt{m(u)}e_r\right),u^{-\beta-2}e_s\ra}}\nonumber\\
& \quad =  \beta(\beta+1)\sum_{r=0}^{\infty}{\lambda_r\sum_{s=0}^{\infty}{\la \nabla\left(\sqrt{m(u)}e_r\right),e_{s}\ra\la \nabla\left(\sqrt{m(u)}e_r\right),u^{-\beta-2}e_s\ra}}\nonumber\\
%& \quad = \{\mbox{sum in }s\} = \beta(\beta+1)\sum_{r}{\lambda_r\la \left|\nabla\left(\sqrt{m(u)}e_r\right)\right|^2,u^{-\beta-2}\ra}\leq C(\beta,\gamma,\{\lambda\}_r)\int_{D}{\left|\nabla\left(\sqrt{m(u)}e_r\right)\right|^2u^{-\beta-2}}\nonumber\\
& \quad =  \beta(\beta+1)\sum_{r=0}^{\infty}{\lambda_r\la \left|\nabla\left(\sqrt{m(u)}e_r\right)\right|^2,u^{-\beta-2}\ra}\label{eq:5000}\\%\leq C(\beta,\{\lambda\}_r)\int_{D}{\left|\nabla\left(\sqrt{m(u)}e_r\right)\right|^2u^{-\beta-2}}\nonumber\\
&\quad \leq C(\beta,\{\lambda\}_{r})\left\{\la m^{-1}(u)(m'(u))^2u^{-\beta-2}\,\nabla u,\nabla u\ra + \int_{D}{m(u)u^{-\beta-2}\m x}\right\}.\label{eq:19}%\nonumber\\
%&\quad = C(\beta,\gamma,\{\lambda\}_r)\left\{\la u^{-\beta+\gamma-4}\nabla u,\nabla u\ra + \la u^{-\beta+\gamma-3},\left|\nabla u\right|\ra + \int_{D}{u^{\gamma-\beta-2}}\right\}.
\end{align}

\begin{rem}
Alternatively, one can identify \eqref{eq:5000} by using \cite[Section 3]{curtain1970ito}.
\end{rem}

\emph{Term $I_3$ for $G=F_2$}. We compute the It\^o correction 
\begin{align}\label{eq:12}
& \frac{1}{2}\mbox{Tr}\left[F_{2,uu}(u)(\Phi(u)Q^{1/2})(\Phi(u)Q^{1/2})^T\right]=\sum_{r=0}^{\infty}{\lambda_r\sum_{z=0}^{\infty}{(1+z^2)\la \sqrt{m(u)}e_r,e_{z,x}\ra^2}}\\
& \quad = \sum_{r=0}^{\infty}{\lambda_r\sum_{z=0}^{\infty}{\la \sqrt{m(u)}e_r,e_{z,x}\ra^2}}+\sum_{r=0}^{\infty}{\lambda_r\sum_{z=0}^{\infty}{z^2\la \sqrt{m(u)}e_r,e_{z,x}\ra^2}}=:T_1+T_2.\nonumber
\end{align}
%\begin{rem}
Once again, the reader can convince oneself of the nature of \eqref{eq:12} by thinking of a finite-dimensional equivalent of the problem, thus relying on the matrices $Q_m$ and $\Phi_{m}(u)$ defined in \eqref{eq:11}, as well as on the matrix $[F_{2,uu}(u)]_m=\mbox{diag}\{(1+z^2)\}_{z=1,\cdots,m}$. See Remark \ref{rem:3} also.

%\end{rem}
We bound $T_2$. Given the nature of the trigonometric basis $\{e_r\}_{r=0}^{\infty}$, we have (for $r\geq 1$), that $re_{r,x}=\delta(r)\Delta e_{\sigma(r)}$, %and $r^2e_r=-\Delta u_r$, 
for some injective function $\sigma:\mathbb{N}\rightarrow \mathbb{N}$ and where $\delta(r)\in\{-1;+1\}$. We use integration by parts and the Parseval identity (for the sum over $z$) and obtain
\begin{align}
T_2 & = \sum_{r=0}^{\infty}{\lambda_r\sum_{z=0}^{\infty}{\la \sqrt{m(u)}e_r,\Delta e_z\ra^2}}=\sum_{r=0}^{\infty}{\lambda_r\sum_{z=0}^{\infty}{\la \Delta\left(\sqrt{m(u)}e_r\right),e_{z}\ra^2}}
%& = \{\mbox{Parseval on }z\}=\sum_{r}{\lambda_r\left\|\Delta\left(\sqrt{m(u)}e_r\right)\right\|^2}=\{\mbox{decay of }\{\lambda_r\}_r\}\nonumber\\
 = \sum_{r=0}^{\infty}{\lambda_r\left\|\Delta\left(\sqrt{m(u)}e_r\right)\right\|^2}\label{eq:5001}\\
& \leq C\sum_{r=0}^{\infty}{\lambda_r\left[\left\| \left\{-\frac{1}{4}m^{-3/2}(m')^2+\frac{1}{2}m^{-1/2}(u)m''(u)\right\}|\nabla u|^2e_r\right\|^2+\left\|\frac{1}{2}m^{-1/2}(u)m'(u)\Delta ue_r\right\|^2\right.}\label{eq:3100}\\
& \quad +\left.\left\|m^{-1/2}(u)m'(u)\nabla u\,e_{r,x}\right\|^2+\left\|\sqrt{m(u)}\Delta e_r\right\|^2\right]\nonumber\\
& \leq C(\{\lambda_r\}_r)\left\{\la\left[ m^{-1}(u)(m''(u))^2+m^{-3}(u)(m'(u))^4\right]\nabla u|^2,|\nabla u|^2\ra \right.\nonumber\\
&\quad + \left.\la m^{-1}(u)(m'(u))^2\Delta u, \Delta u\ra+\la m^{-1}(u)(m'(u))^2\nabla u, \nabla u\ra + \int_{D}{m(u)\m x}\right\},\label{eq:23}
\end{align}
where the right-hand-side of \eqref{eq:5001} can also be inferred from \cite[Section 3]{curtain1970ito}.
%{\color{red}
\begin{rem}\label{rem:2} 
Given the polynomial nature of $m(u)\left|_{A_0\cup A_{\infty}}\right.$, it is easy to notice that the multiplying term $T_3:=-\frac{1}{4}m^{-3/2}(m')^2+\frac{1}{2}m^{-1/2}(u)m''(u)$ in \eqref{eq:3100} vanishes if and only if $\gamma_1=\gamma_2=2$. In all other cases, the terms making up $T_3$ are proportional to each other.
\end{rem}
As for $T_1$, the computation is simpler, and it reads, thanks to the Parseval inequality 
\begin{align}
T_1 & = \sum_{r=0}^{\infty}{\lambda_r\sum_{z=0}^{\infty}{\la \sqrt{m(u)}e_r, e_{z,x}\ra^2}}=\sum_{r=0}^{\infty}{\lambda_r\sum_{z=0}^{\infty}{\la \nabla\left(\sqrt{m(u)}e_r\right),e_{z}\ra^2}}\label{eq:5002}
%& = \{\mbox{Parseval on }z\}=\sum_{r}{\lambda_r\left\|\nabla\left(\sqrt{m(u)}e_r\right)\right\|^2}=\{\mbox{decay of }\{\lambda_r\}_r\}\\
 = \sum_{r=0}^{\infty}{\lambda_r\left\|\nabla\left(\sqrt{m(u)}e_r\right)\right\|^2}\\
& \leq C\sum_{r=0}^{\infty}{\lambda_r\!\left[\left\| m^{-1/2}(u)m'(u)\nabla u\,e_r\right\|^2+\left\|\sqrt{m(u)}e_{r,x}\right\|^2\right]}\\\nonumber
&  \leq C(\{\lambda_r\}_r)\left\{\la m^{-1}(u)(m'(u))^2\nabla u, \nabla u\ra + \int_{D}{m(u)\m x}\right\},\nonumber
\end{align}
%We put everything together and deduce that \eqref{eq:12} can be bounded as
where the right-hand-side of \eqref{eq:5002} can once again be inferred from \cite[Section 3]{curtain1970ito}. We deduce
\begin{align}\label{eq:13}
& \frac{1}{2}\mbox{Tr}\left[F_{2,uu}(u)(\Phi(u)Q^{1/2})(\Phi(u)Q^{1/2})^T\right] \leq C(\{\lambda_r\}_r)\left\{\la\left[ m^{-1}(u)(m''(u))^2+m^{-3}(u)(m'(u))^4\right]|\nabla u|^2,|\nabla u|^2\ra \right.\nonumber\\%C(\{\lambda_r\}_r)\left\{\la m^{-1}(u)(m''(u))^2|\nabla u|^2,|\nabla u|^2\ra \right.\nonumber\\
&\quad \quad+ \left.\la m^{-1}(u)(m'(u))^2\Delta u, \Delta u\ra+\la m^{-1}(u)(m'(u))^2\nabla u, \nabla u\ra + \int_{D}{m(u)\m x}\right\},
%& \quad \leq C(\gamma,\{\lambda_r\}_r)\left\{\la u^{\gamma-4}|\nabla u|^2,|\nabla u|^2\ra+\la u^{\gamma-2}\Delta u, \Delta u\ra+\la u^{\gamma-2}\nabla u, \nabla u\ra + \int_{D}{u^{\gamma}}\right\}.
\end{align} 
\emph{Term $I_4$ for $G=F_1$}. We rely on \cite[Theorem 4.36]{Da-Prato2014a} and bound the It\^o isometry term associated with $I_4$. We use integration by parts and the Parseval identity to deduce
\begin{align}\label{eq:3101}
& \sum_{r=0}^{\infty}{\lambda_r\left|\sum_{z=0}^{\infty}{-\beta\la u^{-\beta-1},e_z\ra\la\nabla\left(\sqrt{m(u)}e_r\right),e_z\ra}\right|^2}=\beta\sum_{r=0}^{\infty}{\lambda_r\left|\la u^{-\beta-1},\nabla\left(\sqrt{m(u)}e_r\right)\ra\right|^2}\nonumber\\
& \quad = \beta\sum_{r=0}^{\infty}{\lambda_r\left|\la (\beta+1)u^{-\beta-2}\nabla u,\sqrt{m(u)}e_r\ra\right|^2}
%& \quad =\beta\sum_{r=0}^{\infty}{\lambda_r\left|\la u^{-\beta-1},\frac{1}{2}m^{-1/2}(u)m'(u)u_xe_r+\sqrt{m(u)}e_{r,x}\ra\right|^2}\nonumber\\
%& \quad \leq C(\{\lambda_r\}_{r=0}^{\infty},\beta)\left\{\la u^{-2(\beta+1)}m^{-1}(u)(m'(u))^2\nabla u,\nabla u \ra+\int_{D}{u^{-2(\beta+1)}m(u)}\right\}. 
\leq C(\{\lambda_r\}_{r},\beta)\la u^{-\beta-2}m(u)\nabla u, u^{-\beta-2}\nabla u\ra.
\end{align}
Given the definition of $\tau_n$, we deduce that $I_4$ is a square-integrable martingale with mean zero, see \cite[Proposition 4.28]{Da-Prato2014a}. The contribution of $I_4$ can thus be neglected. %{\bfseries \emph{ expected value}}.

\emph{Term $I_4$ for $G=F_2$}. Again relying on \cite[Theorem 4.36]{Da-Prato2014a}, we bound the It\^o isometry term associated with $I_2$. Similarly to \eqref{eq:3101}, we deduce
\begin{align}\label{eq:3001}
& \sum_{r=0}^{\infty}{\lambda_r\left|\sum_{z=0}^{\infty}{(1+z^2)\la u,e_z\ra\la\nabla\left(\sqrt{m(u)}e_r\right),e_z\ra}\right|^2}=\sum_{r=0}^{\infty}{\lambda_r\left|\sum_{z=0}^{\infty}{\la u,e_z-\Delta e_z\ra\la\nabla\left(\sqrt{m(u)}e_r\right),e_z\ra}\right|^2}\nonumber\\
& \quad \leq \sum_{r=0}^{\infty}{2\lambda_r\left|\la u,\nabla\left(\sqrt{m(u)}e_r\right)\ra\right|^2}+\sum_{r=0}^{\infty}{2\lambda_r\left|\la \Delta u,\nabla\left(\sqrt{m(u)}e_r\right)\ra\right|^2}\nonumber\\
%& \quad \leq C(\{\lambda\}_{r=0}^{\infty})\left\{\|u\|^2+\la m^{-1}(u)(m'(u))^2\nabla u, \nabla u\ra + \|\Delta u\|^2+\int_{D}{m(u)}\right\}.
& \quad = \sum_{r=0}^{\infty}{2\lambda_r\left|\la \nabla u,\sqrt{m(u)}e_r\ra\right|^2}+\sum_{r=0}^{\infty}{2\lambda_r\left|\la  u_{xxx},\sqrt{m(u)}e_r\ra\right|^2}\nonumber\\
& \quad \leq C(\{\lambda_r\}_{r})\left\{\la \nabla u,m(u)\nabla u\ra + \la u_{xxx},m(u)u_{xxx}\ra \right\}.
\end{align}
In this case, the definition of $\tau_n$ does not imply that $I_4$ is a square-integrable martingale with mean zero. This is due to the presence of the term $\la u_{xxx},m(u)u_{xxx}\ra$.

\subsection{Clustering contributions from the It\^o formula}\label{ss:2}

In the previous section we have provided bounds for the terms $I_2$, $I_3$, $I_4$ associated with the It\^o formula applied to the functionals $F_1(u)$ and $F_2(u)$. These bounds contain terms which can be clustered in five distinct families, identified as 
\begin{align}
& \int_{D}{p(u)}, \label{eq:15}\tag{F1}\\
& \la p(u)\Delta u, \Delta u\ra, \label{eq:16}\tag{F2}\\
& \la p(u)|\nabla u|^2,|\nabla u|^2\ra, \label{eq:17}\tag{F3}\\
& \la p(u)\nabla u, \nabla u\ra, \label{eq:18}\tag{F4}\\
& \la p(u)u_{xxx},u_{xxx}\ra, \label{eq:180}\tag{F5}
\end{align}
for some $p\in\mathcal{C}(0,\infty)$. Notice that all contributions to the It\^o formula are well defined, because of assumption \eqref{eq:3002}. With the exception of the terms in the right-hand-side of \eqref{eq:3101} (associated with the It\^o isometry of $I_4$ for the functional $F_1(u)$), we now cluster all the terms belonging to the same family.
\vspace{0.5 pc}\\
\emph{Terms of kind \eqref{eq:15}}. Relevant terms are gathered from \eqref{eq:13}, \eqref{eq:19}, {\color{fu}\eqref{eq:25a}, \eqref{eq:6}}, %\eqref{eq:3001}, 
adding up to
\begin{align}\label{eq:20}
C(\{\lambda_r\}_r)\int_{D}{m(u)\m x}+C(\{\lambda_r\}_r,\beta)\int_{D}{m(u)u^{-\beta-2}\m x}{\color{fu}+\la g(u),u \ra+\la g(u),-\beta u^{-\beta-1}\ra}.%+C(\{\lambda\}_{r=0}^{\infty})\la u, u\ra.
\end{align} 
\emph{Terms of kind \eqref{eq:16}}. Relevant terms are gathered from \eqref{eq:5}, \eqref{eq:21}, \eqref{eq:13}, %\eqref{eq:3001}, 
adding up to
\begin{align}\label{eq:22}
& -\beta(\beta+1)\la \Delta u, m(u)u^{-\beta-2}\Delta u \ra-\la \Delta u, m(u)\Delta u\ra-\la\Delta u, m(u)W''(u)\Delta u\ra\nonumber\\
 & \quad + C(\{\lambda_r\}_r)\la m^{-1}(u)(m'(u))^2\Delta u, \Delta u\ra.%+ C(\{\lambda\}_{r=0}^{\infty})\la \Delta u, \Delta u\ra.
\end{align} 
\emph{Terms of kind \eqref{eq:17}}. Relevant terms are gathered from \eqref{eq:5}, \eqref{eq:21}, \eqref{eq:25}, \eqref{eq:13}, adding up to
\begin{align}\label{eq:24}
& C(\beta)\la (m(u)u^{-\beta-2})''|\nabla u|^2,|\nabla u|^2\ra + \frac{1}{3}\la m''(u)|\nabla u|^2,|\nabla u|^2\ra +\frac{1}{3}\la (m(u)W''(u))''|\nabla u|^2,|\nabla u|^2\ra\nonumber\\
& \quad + \frac{1}{3}\la h'(u)|\nabla u|^2, |\nabla u|^2\ra + C(\{\lambda_r\}_r)\la\left[ m^{-1}(u)(m''(u))^2+m^{-3}(u)(m'(u))^4\right]|\nabla u|^2,|\nabla u|^2\ra.%+ C(\{\lambda_r\}_r)\la m^{-1}(u)(m''(u))^2|\nabla u|^2,|\nabla u|^2\ra.
\end{align} 
%where $f(u)=m(u)W''(u)$.\vspace{0.5 pc}\\
\emph{Terms of kind \eqref{eq:18}}. Relevant terms are gathered from \eqref{eq:5}, \eqref{eq:6}, \eqref{eq:21}, \eqref{eq:25}, \eqref{eq:19}, \eqref{eq:13}, {\color{fu}\eqref{eq:25a}}, \eqref{eq:3001}, adding up to
\begin{align}\label{eq:26}
& -\beta(\beta+1)\la W''(u)\nabla u,m(u)u^{-\beta-2}\nabla u\ra -C(\beta)\la h(u)|\nabla u|^2,u^{-\beta-1}\ra -\la W''(u)\nabla u,m(u)\nabla u\ra\nonumber\\
& \quad +\la h(u)u, |\nabla u|^2\ra + C(\{\lambda\}_r)\la m^{-1}(u)(m'(u))^2u^{-\beta-2}\nabla u,\nabla u\ra + C(\{\lambda\}_{r})\la m^{-1}(u)(m'(u))^2\nabla u, \nabla u\ra\nonumber\\
& \quad +{\color{fu}\la g'(u)\nabla u, \nabla u\ra} + C(\{\lambda\}_r)\la \nabla u,m(u)\nabla u \ra.
\end{align}
\emph{Terms of kind \eqref{eq:180}}. Relevant terms are gathered from \eqref{eq:21}, \eqref{eq:3001}, adding up to 
\begin{align}\label{eq:3102}
(C(\{\lambda_r\}_{r})-1)\la m(u)u_{xxx},u_{xxx}\ra.
\end{align}

\subsection{Parameter tuning on $A_0\cup A_{\infty}$}\label{ss:3}

We now look for conditions on the parameters $p,B_h,p_h,c_h,{\color{fu}B_g,p_g,c_g},\gamma_1,\gamma_2,$ and $\{\lambda_r\}_{r=0}^{\infty}$ in order to obtain \eqref{eq:50}. More specifically, we look for conditions on these parameters in such a way that some of the terms in \eqref{eq:20}, \eqref{eq:22}, \eqref{eq:24}, \eqref{eq:26}, and \eqref{eq:3102} can be bounded by the two Gronwall type terms $\int_{D}{u^{-\beta}}$ and $\|u\|_{H^1}^2$, while the remaining can be bounded by constants. 
%{\color{red} explain more about Gronwall}.\\
In order to easily identify the relevant parameters, for each of the families {\color{fu}\eqref{eq:15}}--\eqref{eq:18} we draw two summary tables. As for the first table:
\begin{description}
\item (i) each column is associated with a term of the family in question, the terms being listed in order of appearance in the corresponding expression among \eqref{eq:20}, \eqref{eq:22}, \eqref{eq:24}, and \eqref{eq:26}. 
\item (ii) the second row shows the degree of the monomial restriction $p(u)\left|_{A_0}\right.$. %We do not write down derivatives of $u$, as these already identify the family which the matrix is describing. 
\item (iii) the first row shows the constants multiplying $p(u)\left|_{A_0}\right.$.
\end{description}
We will denote this kind of table by $\mathcal{A}_{0}$. As for the second table, everything is defined in the same way, but with the region $A_0$ replaced by $A_{\infty}$. We will denote this kind of table by $\mathcal{A_\infty}$. We deal with the analysis on the region $A_1$ in the following subsection.
\vspace{0.5 pc}\\
{\color{fu}
\emph{Summary table and conditions for family \eqref{eq:15}}. Tables $\mathcal{A}_0$ and $\mathcal{A}_{\infty}$ summarising \eqref{eq:20} are given in Figure \ref{f:1a}.
\begin{figure}[htbp]
     \begin{center}
     $\mathcal{A}_0=$
     \begin{tabular}{|L|L|L|L|}
     \hline
     C(\{\lambda_r\}) & C(\beta,\{\lambda_r\}) & B_g & -\beta B_g \\
     \hline
    \gamma_1 & \gamma_1-\beta-2 & -p_g+1 & -\beta-1-p_g\\
    \hline
     \end{tabular}
     \end{center}
    % \caption{Esempio d?ambiente \texttt{table}}
     %\label{1}
%\end{table}
%\begin{table}[htbp]
     \begin{center}
     $\mathcal{A}_\infty=$
     \begin{tabular}{|L|L|L|L|}
     \hline
      C(\{\lambda_r\}) & C(\beta,\{\lambda_r\}) & -B_g & \beta B_g \\
     \hline
    \gamma_2 & \gamma_2-\beta-2 & c_g+1 & -\beta-1+c_g\\
    \hline
     \end{tabular}
     \end{center}
     \caption{Tables $\mathcal{A}_0$ and $\mathcal{A}_\infty$ for family \eqref{eq:15}.}
     \label{f:1a}
\end{figure}
%$$
%     {\mathcal A}_{0} = \left(
%     \begin{array}{cccc}
%     C(\{\lambda_r\}) & C(\beta,\{\lambda_r\}) & B_g & -\beta B_g \\%& C(\{\lambda\}_{r=0}^{\infty})\\
%     \gamma_1 & \gamma_1-\beta-2 & -p_g+1 & -\beta-1-p_g %& 2\\
%     \end{array}
%     \right),
%$$
%$$
%     {\mathcal A}_{\infty} = \left(
%     \begin{array}{cccc}
%     C(\{\lambda_r\}) & C(\beta,\{\lambda_r\}) & -B_g & \beta B_g \\%& C(\{\lambda\}_{r=0}^{\infty})\\
%     \gamma_2 & \gamma_2-\beta-2 & c_g+1 & -\beta-1+c_g %& 2\\
%     \end{array}
%     \right).
%$$
%Column 5 is associated with a Gronwall term, and can thus be neglected. As for the remaining columns, 
Condition \eqref{eq:45} insures that the leading polynomial order is contained in the fourth (respectively, third) column for $\mathcal{A}_{0}$ (respectively, $\mathcal{A}_{\infty}$). The contribution given by the family \eqref{eq:15} is then bounded by a constant.} %and a Gronwall term.}
\vspace{0.5 pc}\\
\emph{Summary table and conditions for family \eqref{eq:16}}. Tables $\mathcal{A}_0$ and $\mathcal{A}_{\infty}$ summarising \eqref{eq:22} are given in Figure \ref{f:1b}.
\begin{figure}[htbp]
     \begin{center}
     $\mathcal{A}_0=$
     \begin{tabular}{|L|L|L|L|}
     \hline
    -C(\beta) & -1 & -p(p+1) & \gamma_1^2C(\{\lambda_r\}_r) \\
     \hline
     \gamma_1-\beta-2 & \gamma_1 & \gamma_1-p-2 & \gamma_1-2\\
    \hline
     \end{tabular}
     \end{center}
    % \caption{Esempio d?ambiente \texttt{table}}
     %\label{1}
%\end{table}
%\begin{table}[htbp]
     \begin{center}
     $\mathcal{A}_\infty=$
     \begin{tabular}{|L|L|L|L|}
     \hline
       -C(\beta) & -1 & -p(p+1) & \gamma_2^2C(\{\lambda_r\}_r) \\
     \hline
     \gamma_2-\beta-2 & \gamma_2 & \gamma_2-p-2 & \gamma_2-2\\
    \hline
     \end{tabular}
     \end{center}
     \caption{Tables $\mathcal{A}_0$ and $\mathcal{A}_\infty$ for family \eqref{eq:16}.}
     \label{f:1b}
\end{figure}
% $$
%     {\mathcal A}_{0} = \left(
%     \begin{array}{cccc}
%     -C(\beta) & -1 & -p(p+1) & \gamma_1^2C(\{\lambda_r\}_r) \\%& C(\{\lambda\}_{r=0}^{\infty}) \\
%     \gamma_1-\beta-2 & \gamma_1 & \gamma_1-p-2 & \gamma_1-2 %& 0\\
%     \end{array}
%     \right),
%$$
%$$
%     {\mathcal A}_{\infty} = \left(
%     \begin{array}{cccc}
%     -C(\beta) & -1 & -p(p+1) & \gamma_2^2C(\{\lambda_r\}_r) \\%& C(\{\lambda\}_{r=0}^{\infty})\\
%     \gamma_2-\beta-2 & \gamma_2 & \gamma_2-p-2 & \gamma_2-2 %& 0\\
%     \end{array}
%     \right).
%$$
%Parameters $\beta,\gamma_1,\gamma_2$ have so far been fixed. 
For both $\mathcal{A}_0$ and $\mathcal{A}_{\infty}$, the only positive contribution comes from column 4. This contribution can be compensated, e.g., with column 1 (in the case of $\mathcal{A}_{0}$) or column 2 (in the case of $\mathcal{A}_{\infty}$) by using \eqref{eq:43}. %As for $\mathcal{A}_{\infty}$, the only positive contribution comes from columns 4. This contribution can be compensated, for example, with column 2, again using \eqref{eq:43}.
\vspace{0.5 pc}\\
\emph{Summary table and conditions for family \eqref{eq:17}}.
Tables $\mathcal{A}_0$ and $\mathcal{A}_{\infty}$ summarising \eqref{eq:24} are given in Figure \ref{f:1c}.
\begin{figure}[htbp]
     \begin{center}
     $\mathcal{A}_0=$
     \begin{tabular}{|L|L|L|L|L|}
     \hline
    C(\gamma_1,\beta) & C(\gamma_1) & p(p+1)(\gamma_1-p-2)(\gamma_1-p-3) & -p_hB_h/3 & C(\gamma_1)C(\{\lambda_r\}_r) \\
     \hline
    \gamma_1-\beta-4 & \gamma_1-2 & \gamma_1-p-4 & -p_h-1 & \gamma_1-4 \\
    \hline
     \end{tabular}
     \end{center}
    % \caption{Esempio d?ambiente \texttt{table}}
     %\label{1}
%\end{table}
%\begin{table}[htbp]
     \begin{center}
     $\mathcal{A}_\infty=$
     \begin{tabular}{|L|L|L|L|L|}
     \hline
       C(\gamma_2,\beta) & C(\gamma_2) & p(p+1)(\gamma_2-p-2)(\gamma_2-p-3) & -c_hB_h/3 & C(\gamma_2)C(\{\lambda_r\}_r) \\
     \hline
     \gamma_2-\beta-4 & \gamma_2-2 & \gamma_2-p-4 & c_h-1 & \gamma_2-4 \\
    \hline
     \end{tabular}
     \end{center}
     \caption{Tables $\mathcal{A}_0$ and $\mathcal{A}_\infty$ for family \eqref{eq:17}.}
     \label{f:1c}
\end{figure}
%$$
%     {\mathcal A}_{0} = \left(
%     \begin{array}{ccccc}
%     C(\gamma_1,\beta) & C(\gamma_1) & p(p+1)(\gamma_1-p-2)(\gamma_1-p-3) & -p_hB_h/3 & C(\gamma_1)C(\{\lambda_r\}_r)\\
%     \gamma_1-\beta-4 & \gamma_1-2 & \gamma_1-p-4 & -p_h-1 & \gamma_1-4\\
%     \end{array}
%     \right),
%$$
%$$
%     {\mathcal A}_{\infty} = \left(
%     \begin{array}{ccccc}
%     C(\gamma_2,\beta) & C(\gamma_2) & p(p+1)(\gamma_2-p-2)(\gamma_2-p-3) & -c_hB_h/3 & C(\gamma_2)C(\{\lambda_r\}_r)\\
%     \gamma_2-\beta-4 & \gamma_2-2 & \gamma_2-p-4 & c_h-1 & \gamma_2-4\\
%     \end{array}
%     \right).
%$$
For $\mathcal{A}_{0}$ (respectively, $\mathcal{A}_{\infty}$) we can pick $p_h,B_h$ big enough (respectively, $c_h,B_h$ big enough) so that column 4 contains the leading-order monomial, with also sufficiently big multiplicative constant. Thus column 4 compensates all the other columns. We have thus invoked \eqref{eq:44}.\vspace{0.5 pc}\\%{\color{red} Note that we haven't required anything from the potential $W$}. 
%The condition suggested by this family is then simply
%\begin{align}\label{eq:44}
%\tag{Con5}
%p',B',c'\mbox{ big enough}.
%\end{align}
\emph{Summary table and conditions for family \eqref{eq:18}}.
Tables $\mathcal{A}_0$ and $\mathcal{A}_{\infty}$ summarising \eqref{eq:26} are given in Figure \ref{f:1d}.
\begin{figure}[htbp]
     \begin{center}
     $\mathcal{A}_0=$
     \begin{tabular}{|L|L|L|L|L|L|L|L|}
     \hline
     -C(\beta)p(p+1) & -C(\beta)B_h & -p(p+1) & B_h & C(\{\lambda_r\}_r,\gamma_1) & C(\{\lambda_r\}_r,\gamma_1) & {\color{fu}-B_gp_g} & 1 \\
     \hline
     \gamma_1-\beta-p-4 & -p_h-\beta-1 & \gamma_1-p-2 & -p_h+1 & \gamma_1-\beta-4 & \gamma_1-2 & {\color{fu}-p_g-1} & \gamma_1 \\
    \hline
     \end{tabular}
     \end{center}
    % \caption{Esempio d?ambiente \texttt{table}}
     %\label{1}
%\end{table}
%\begin{table}[htbp]
     \begin{center}
     $\mathcal{A}_\infty=$
     \begin{tabular}{|L|L|L|L|L|L|L|L|}
     \hline
      -C(\beta)p(p+1) & C(\beta)B_h & -p(p+1) & -B_h & C(\{\lambda_r\}_r,\gamma_2) & C(\{\lambda_r\}_r,\gamma_2) & {\color{fu}-B_gc_g} & 1 \\
     \hline
      \gamma_2-\beta-p-4 & c_h-\beta-1 & \gamma_2-p-2 & c_h+1 & \gamma_2-\beta-4 & \gamma_2-2 & {\color{fu}c_g-1} & \gamma_2 \\
    \hline
     \end{tabular}
     \end{center}
     \caption{Tables $\mathcal{A}_0$ and $\mathcal{A}_\infty$ for family \eqref{eq:18}.}
     \label{f:1d}
\end{figure}

%$$
%     {\mathcal A}_{0} = \left(
%     \begin{array}{cccccccc}
%     -C(\beta)p(p+1) & -C(\beta)B_h & -p(p+1) & B_h & C(\{\lambda_r\}_r,\gamma_1) & C(\{\lambda_r\}_r,\gamma_1) & {\color{fu}-B_gp_g} & 1\\
%     \gamma_1-\beta-p-4 & -p_h-\beta-1 & \gamma_1-p-2 & -p_h+1 & \gamma_1-\beta-4 & \gamma_1-2 & {\color{fu}-p_g-1} & \gamma_1\\
%     \end{array}
%     \right),
%$$
%$$
%     {\mathcal A}_{\infty} = \left(
%     \begin{array}{cccccccc}
%     -C(\beta)p(p+1) & C(\beta)B_h & -p(p+1) & -B_h & C(\{\lambda_r\}_r,\gamma_2) & C(\{\lambda_r\}_r,\gamma_2) & {\color{fu}-B_gc_g} & 1\\
%     \gamma_2-\beta-p-4 & c_h-\beta-1 & \gamma_2-p-2 & c_h+1 & \gamma_2-\beta-4 & \gamma_2-2 & {\color{fu}c_g-1} & \gamma_2\\
%     \end{array}
%     \right).
%$$

%As for $\mathcal{A}_0$, we can pick $p'$ big enough so that column 2 contains the leading-order monomial. Even if we pick $B'$ big enough, this might not be enough to compensate column 4, whose multiplicative constant also contains $B'$. As a result, we may bound all the contributions with a positive constant, rather than 0. This is acceptable, because family \eqref{eq:18} is associated with $\|\nabla u\|^2$, and such term is a Gronwall term. An analogous consideration applies to $\mathcal{A}_{\infty}$, with the roles of columns 2 and 4 swapped. We have invoked \eqref{eq:44} again.
If we invoke \eqref{eq:45} for both $\mathcal{A}_0$ and $\mathcal{A}_{\infty}$, then column 7 contains the leading order. Thus all other columns are compensated by a constant.
\vspace{0.5 pc}\\
\emph{Conditions for family \eqref{eq:180}}. Contribution \eqref{eq:3102} is negative as long as we invoke \eqref{eq:43}.

\subsection{Parameter tuning on $A_{1}$}\label{ss:4}

Conditions \eqref{eq:43}-\eqref{eq:45} are also enough to provide the same conclusions, as in Subsection \ref{ss:3}, for the families \eqref{eq:15}-\eqref{eq:180} analysed over $A_1$. More specifically: the domain $D$ being bounded, the continuity of $m$ does not alter the estimate for the family \eqref{eq:15}; the estimate for the family \eqref{eq:16} still holds due to \eqref{eq:43} and \eqref{eq:1000c}; 
the estimate for the family \eqref{eq:17} still holds due to \eqref{eq:1000c}--\eqref{eq:1000b} and \eqref{eq:44}; the estimate for the family \eqref{eq:18} still holds, due to \eqref{eq:1000c} and the fact that we are allowed to bound everything with a constant times $|\nabla u|^2$, so there is no issue in the local behaviour in a neighbourhood of $u=1$. Finally, thanks to \eqref{eq:1000c}, nothing needs to be adapted for the family \eqref{eq:180}. \vspace{0.5 pc}\\
We can complete the proof of Proposition \ref{p:1} by taking the expected value in the It\^o formula \eqref{eq:4} for $G=F_1+F_2$.

\section{Analysis of results}\label{s:4}

We compare our setting to that of J. Fischer and F. Gr\"un, whose paper \cite{Fischer2018a} has inspired us to this work. In \cite{Fischer2018a}, existence of a $\mathbb{P}$-a.s. positive solution to the conservative thin-film equation (i.e., equation \eqref{eq:1} with $h\equiv g\equiv 0$) is established in the case of quadratic mobility $m(u)=u^2$. This specific mobility, corresponding to $\gamma_1=\gamma_2=2$ in our notation, results in a linear stochastic noise which makes $h$ and $g$ unnecessary in the argument.
%This specific choice of mobility results in no need of potentials $h$, ${\color{fu}g}$. 
We detail this last statement by making a direct comparison to our computations.
\vspace{0.5 pc}\\
%Their argument proceeds via a Galerkin scheme, which heavily relies on the linearity of the noise. In particular, there is no term belonging to family \eqref{eq:17} in their argument. This is due to the following reasons:
\emph{No need for $h$}. No term belonging to the family \eqref{eq:17} arises when $\gamma_1=\gamma_2=2$. Firstly, the It\^o correction applied to $\|\nabla u\|^2$ does not produce any such term, because of the linear nature of $\sqrt{m(u)}=u$, see Remark \ref{rem:2}. We can thus drop the \eqref{eq:17}-term in \eqref{eq:13}, which corresponds to column 5 in $\mathcal{A}_0$ and $\mathcal{A}_{\infty}$. Secondly, if one picks $p:=\beta>2$ (this is compatible with the setting in \cite{Fischer2018a}), some computations can be performed better. In particular, one can combine the drift contributions coming from the It\^o formula applied to functional $F_3(u):=\|\nabla u\|^2+F_1(u)$, thus getting, for $p_r:=-\Delta u+W'(u)$ 
\begin{align*}
& \la u_x,\nabla(-\nabla(m(u)\nabla(-p_r)))\ra+\la W'(u),-\nabla(m(u)\nabla(-p_r))\ra\\
& \quad = \la \Delta u, \nabla(m(u)\nabla(-p_r))\ra + \la \nabla [W'(u)],m(u)\nabla(-p_r) \ra \\
& \quad= -\la \nabla[\Delta u], m(u)\nabla(-p_r)\ra + \la \nabla[W'(u)],m(u)\nabla(-p_r) \ra = -\la p_{r,x},m(u)p_{r,x}\ra \leq 0.
\end{align*}
The above computation is a way of regrouping relevant drift terms in a slightly differently way. More specifically, the final term $\la p_{r,x},m(u)p_{r,x}\ra$ can be rewritten as
\begin{align*}
\la m(u)u_{xxx},u_{xxx}\ra+\la W''(u)\nabla u,m(u)W''(u)\nabla u\ra - 2\la u_{xxx},m(u)W''(u)\nabla u\ra
\end{align*}
and the last term in above expression contains the contributions of columns 1 and 3 of $\mathcal{A}_{0}$ and $\mathcal{A}_{\infty}$ (which coincide, as $\beta=p$, see \eqref{eq:5} and \eqref{eq:21}).
%and in this way the sum of columns 1 and 3 of $\mathcal{A}_0$ and $\mathcal{A}_{\infty}$ can be neglected. 
Finally, column 2 of $\mathcal{A}_0$ and $\mathcal{A}_{\infty}$ is dealt with by not computing the It\^o formula for $\|u\|^2$ at all, as one relies on Poincar\'e inequality arguments based on the conservation of mass. %{\bfseries \emph{ comment}}. 
One is then left only with column 4 of $\mathcal{A}_0$ and $\mathcal{A}_{\infty}$, which are associated with $h$.
%\vspace{0.5 pc}\\
\begin{rem}
In \cite{Fischer2018a}, the quantity $-\la p_{r,x},m(u)p_{r,x}\ra$ is used to balance the It\^o isometry term coming from the stochastic noise given by a suitable combination of $F_1$ and $F_2$. In this paper, we have analysed $F_1$ and $F_2$ separately, thus the quantity $-\la p_{r,x},m(u)p_{r,x}\ra$ has not quite emerged.
\end{rem}
\emph{No need for ${\color{fu}g}$}. This follows under the weaker assumptions $\gamma_2\leq 2$, $2\leq \gamma_1\leq 2+\beta$. The first term in \eqref{eq:20} is of Gronwall type, simply because
\begin{align*}
\int_{D}{m(u)\m x}\leq C+\|u\|^{\gamma_2}_{L^{\gamma_2}}\leq C + C\|u\|_{H^1}^2.
\end{align*} 
As for the second term in \eqref{eq:20}, it is also of Gronwall type. We write
\begin{align*}
C(\{\lambda_r\}_r)\int_{D}{m(u)u^{-\beta-2}\m x}\leq C(\{\lambda_r\}_r)\int_{D}{u^{\gamma_1-\beta-2}\m x}+C(\{\lambda_r\}_r)\int_{D}{u^{\gamma_2-\beta-2}\mathbf{1}_{u>1+\ep}\m x}+C.
\end{align*}
This yields %Because of assumption $2\leq \gamma_1\leq 2+\beta$, we deduce
\begin{align*}
C(\{\lambda_r\}_r)\int_{D}{m(u)u^{-\beta-2}\m x}\leq C(\{\lambda_r\}_r)\int_{D}{u^{\gamma_1-\beta-2}\m x}+C.
\end{align*}
For $2\leq\gamma_1<\beta+2$ and $\beta>2$ we get that $-\beta/(\gamma_1-\beta-2)\geq 1$. We use the H\"older inequality to obtain
\begin{align*}
C(\{\lambda_r\}_r)\int_{D}{u^{\gamma_1-\beta-2}\m x}\leq C(\{\lambda_r\}_r)\left(\int_{D}{u^{-\beta}\m x}\right)^{\frac{\gamma_1-\beta-2}{-\beta}}\leq C(\{\lambda_r\}_r)\int_{D}{u^{-\beta}\m x}+C.
\end{align*} 
When $\gamma_1=\beta+2$, the above inequality is also trivially valid. %Combining the two cases justifies \eqref{eq:42}.
This means that columns 1 and 2 of $\mathcal{A}_0$ and $\mathcal{A}_{\infty}$ for the family \eqref{eq:15} are bounded by Gronwall terms, and ${\color{fu}g}$ is thus superfluous. 
%The above estimate may only be used when the H\^older inequality is applicable. In this particular case, we requite that $\frac{-\beta}{\gamma_1-\beta-2}\geq 1$. Since $\beta>2$ because of \eqref{eq:42}, this immediately gives the condition
%\begin{align}\label{eq:40}
%\tag{Con3}
%2\leq\gamma_1\leq2+\beta.
%\end{align}
\begin{rem}
It is worth noticing that, in the conservative case with quadratic mobility, the potential $W$ is actually needed. The potential $W$ is only involved in bounding all the terms in family \eqref{eq:18}, while it is not necessary to deal with the families \eqref{eq:15}, \eqref{eq:16}, \eqref{eq:17}, and \eqref{eq:180}. In the non-conservative case with mobility $m(u)$ not being quadratic, the use of $W$ can be bypassed by properly tuning $h$, which is needed for the family \eqref{eq:17} anyway. As a matter of fact, we can not use $W$ only, and we may actually not use it at all, as $h$ carries the leading order. 
\end{rem}
%\begin{description}
%\item (i) The It\^o correction, applied to $\|\nabla u\|^2$, does not produce any such term, because on the linear nature of $\sqrt{m(u)}=u$. One can easily convince that the same would happen in our computations, if $m(u)$ was $u^2$. This is precisely what Remark \ref{rem:2} tells us, and we can then drop the \eqref{eq:17}-term of \eqref{eq:23} (which corresponds to column 5 in $\mathcal{A}_0$ and $\mathcal{A}_{\infty}$).
%\item (ii) Fischer and Gr\"un choose $p:=\beta>2$. By doing this, one can actually perform some computations better. In particular, one can combine the drift contributions coming from the It\^o formula applied to functional $F_3(u):=\|\nabla u\|^2+\int_{D}{W(u)}$, thus getting (for $p:=-\Delta u+W'(u)$)
%\begin{align*}
%& \la u_x,\nabla(-\nabla(m(u)\nabla(-p)))\ra+\la W'(u),-\nabla(m(u)\nabla(-p))\ra\\
%& \quad = \la \Delta u, \nabla(m(u)\nabla(-p))\ra + \la [W'(u)]',m(u)\nabla(-p) \ra\\
%& \quad = -\la [\Delta u]', m(u)\nabla(-p)\ra + \la [W'(u)]',m(u)\nabla(-p) \ra\\
%& \quad = -\la p_x,m(u)p_x\ra \leq 0.
%\end{align*}
%This makes sure that sum of columns 1 and 3 in $\mathcal{A}_0$ and $\mathcal{A}_{\infty}$ is negative. 
%\item (iii)
%Column 2 is dealt with by not computing the It\"o formula for $\|u\|^2$ at all. This is possible because their equation is conservative ($h=0$), so one can use Poincar\'e-type inequalities, which only involve $\|\nabla u\|^2$.
%\item (iv) Column 4 (which is that associated with the presence of $h$) is then not necessary for Fischer and Gr\"n, all other columns summing to a negative quantity.
%\end{description}
The contents of this section have shown that the potential $h$ is concerned with addressing nonlinearities of the stochastic noise of \eqref{eq:1} (i.e., analysis for $\gamma_1\neq 2$ or $\gamma_2\neq 2$), while ${\color{fu}g}$ is concerned with being able to deal with noise of ``large'' size in regimes of both low and high density $u$ (i.e., analysis for $\gamma_1<2$ and $\gamma_2>2$). In particular, the terms $h(u)|\nabla u|^2$ and ${\color{fu}g(u)}$ appear to be a plausible drift correction for the specific case of the Dean-Kawasaki model in \eqref{eq:2000}, which corresponds to $\gamma_1=\gamma_2=1$. \vspace{0.5 pc}\\  
\section{Considerations on a Galerkin discretisation of the problem}\label{s:5}

In this work we have dealt with an \emph{a priori} regularity analysis for solutions to \eqref{eq:1}. More specifically, we have assumed the existence of a local regular solution to \eqref{eq:1}, and we have shown that it can be extended up to any given time $T>0$ while also being positive $\mathbb{P}$-a.s. We devote this section to explaining the major difficulties one encounters when trying to prove existence of local solutions to \eqref{eq:1} in the conservative case (corresponding to $h\equiv 0$, ${\color{fu}g\equiv 0}$). 
\vspace{0.5 pc}\\
One may rely on a Galerkin scheme for a spatial discretisation of the problem. Two natural basis choices come up: (i) the trigonometric basis, see Subsection \ref{ss:10}; (ii) the hat basis for the space of periodic linear finite elements on the uniform grid $\{0,h,2h,\cdots,2\pi-h,2\pi\}$, where $h$ in an integer fraction of $2\pi$, see \cite{Fischer2018a}. \vspace{0.5 pc}\\
The use of the trigonometric basis might seem slightly more suitable to deal with the differential operators of \eqref{eq:1}. However, it is subject to a disadvantage. %Keeping Remark \ref{rem:2} in mind, 
For $m:=2\pi h^{-1}$, let $u_m$ be the solution to the $m$-dimensional Galerkin approximation of \eqref{eq:1} with respect to an $L^2$-projection onto $V_m:=\{e_1,\cdots,e_m\}$. It is not hard to see that computing the It\^o formula for the functional $F(u_m)$, where $F$ is the same as in Proposition \ref{p:1}, leads to a few terms carrying a projection operator $\pi_m$ onto $V_m$. In particular, one gets such a projection for the drift component associated with $F_1$. This is an issue, as having projections on the drift term annihilates the compensation that such term could potentially have on the positive terms coming from the It\^o correction for $F_1$ and $F_2$. One can avoid the appearance of such projections by only considering quadratic quantities in $u_m$, such as $F_2(u_m)$. However, one loses any indication of positivity of the solutions $u_m$, which may only be defined up to a random time $\tau$; this is primarily due to the function $W$ not being bounded near the origin, thus preventing us from using the standard existence theory (see, e.g., \cite[Chapter IV, Theorem 2.2]{ikedawatanabe}). 
%
%
%
%
%{\color{red} compactness $\int{|u||\nabla u|^2}$}. 
%
%
%
%
One can not get around this issue by simply smoothening the potential $W$ near the origin, as to do so would not provide uniform estimates for $\mean{F(u_m)}$; one can intuitively see this from the summary tables given in Subsection \ref{ss:3}.
\vspace{0.5 pc}\\
On the other hand, the use of the hat basis proved to be successful in \cite{Fischer2018a} in the case of quadratic mobility. We limit ourselves to briefly summarising the two main reasons for this. Firstly, one may rely on the so called \emph{entropy consistency} for the discrete mobility \cite{grunrumpf}, which allows to discretise the quadratic mobility in a piecewise constant function, for the benefit of relevant integral equations and of projection purposes onto the finite-dimensional Galerkin approximation space.  
Secondly, the solution $u_m$ being piecewise linear, it has piecewise constant derivative $u_{m,x}$. This fact allows to detach contributions involving the quadratic term $\left|u_{m,x}\right|^2$ from the contribution given by the (nonlinear) term $W''(u_m)$, thus simplifying the analysis. Moreover, the contribution given by $W''(u_m)$ is in turn provided by the hat basis spatial discretisation of the problem, which allows to suitably bound the ratios of the values of $u_m$ at adjacent grid nodes. % We do not detail this any further.
These key observations allow the authors in \cite{Fischer2018a} to effectively deal with the nonlinearities of the problem, represented by the quadratic mobility and polynomial potential $W$, within the framework of a Galerkin scheme associated with both positivity and appropriate tightness arguments for the solutions $u_m$. However, this Galerkin approximation scheme does not seem to be extendable (at least in the conservative case) to mobilities whose square roots have unbounded first derivatives, i.e., in which either $\gamma_1<2$ or $\gamma_2>2$. One can find a justification of the previous statement by keeping in mind our discussion for the need of $h$ and $g$ given in Section \ref{s:4}.
%{\color{red}
%%Our method does not yet allow us to set up a Galerkin scheme to prove existence of positive solutions to \eqref{eq:1}. This is due to difficulties associated with projecting nonlinear quantities onto finite-dimensional subspaces, while trying to get meaningful cancellations in the It\^o formula. However, we get some interesting insight in the regularity analysis for \eqref{eq:1}. This insight is, so far, fully consistent with \cite{fischer2018existence}.
%\begin{itemize}
%\item non-linearities do not provide negative contributions.
%\item even if we analyse existence with $|u|^{\gamma}$, and then try to prove positivity, we need to smooth out potentials. doesn't work.
%\item in fischer grun, positivity and existence at same time, no such issue.
%\end{itemize}
%}

\section{Conclusions}\label{ss:6}

For equation \eqref{eq:1}, non-conservative contributions $h$ and ${\color{fu}g}$ appear to be necessary in order to show a priori positivity of solutions in the case of non-quadratic mobility $m$. The role of $h$ is to compensate for nonlinearities arising from the It\^o calculus associated with relevant functionals of the unknown process $u$, while the role of ${\color{fu}g}$ is to compensate for large noise in low and high density regimes. In particular, the Dean-Kawasaki model seems to require a drift correction. The a priori positivity analysis has been performed by using a functional representation with respect to the trigonometric basis of $L^2$. Establishing existence of local solutions (which could then be extended up to any time $T>0$ while preserving positivity) seems to be unpractical if one is to use a Galerkin approximation scheme with respect to this basis; in the conservative case, there seems to be a good chance to prove existence of positive solutions with a Galerkin scheme with respect to the hat basis, but only in the case of mobilities whose square roots have bounded first derivatives ($\gamma_1>2$ and $\gamma_2<2$). If one is to consider different ranges of $\gamma_1$ and $\gamma_2$, then non-conservative corrections could be of use within the hat basis discretisation framework.  
%{\color{red}
%\begin{itemize}
%\item comments of basis (hat or laplace). For big noise, they should not work just the same.
%\item for small noise, hat-Galerkin scheme should be alright for $\gamma_1>2,\gamma_2<2$ (conservative case).
%\item we haven't tried discretising the non-conservative equations. 
%\end{itemize}
%}
\paragraph{Acknowledgements} The Author is supported by a scholarship from the EPSRC Centre for Doctoral Training in Statistical Applied
Mathematics at Bath (SAMBa), under the project EP/L015684/1. The Author wishes to thank his Ph.D. supervisors Johannes Zimmer and Tony Shardlow for their valuable suggestions and constant guidance.  %JZ gratefully acknowledges funding by a Royal Society Wolfson
%Research Merit Award and the Leverhulme Trust for its support via grant RPG-2013-261. 

%\bibliographystyle{plainnat-my}
%\bibliographystyle{plainnat-jz}
% Syntax for ln: ln -s ~/research/biblio/johannes/jz.bib jz.bib
%\bibliography{mbib}

\def\cprime{$'$} \def\cprime{$'$} \def\cprime{$'$}
  \def\polhk#1{\setbox0=\hbox{#1}{\ooalign{\hidewidth
  \lower1.5ex\hbox{`}\hidewidth\crcr\unhbox0}}} \def\cprime{$'$}
  \def\cprime{$'$}

%\def\cprime{$'$} \def\cprime{$'$} \def\cprime{$'$}
%  \def\polhk#1{\setbox0=\hbox{#1}{\ooalign{\hidewidth
%  \lower1.5ex\hbox{`}\hidewidth\crcr\unhbox0}}} \def\cprime{$'$}
%  \def\cprime{$'$}
%
%%\begin{thebibliography}{30}
%\providecommand{\natexlab}[1]{#1}
%\providecommand{\url}[1]{\texttt{#1}}
%\expandafter\ifx\csname urlstyle\endcsname\relax
%  \providecommand{\doi}[1]{doi: #1}\else
%  \providecommand{\doi}{doi: \begingroup \urlstyle{rm}\Url}\fi

\end{document}